\documentclass[11pt, fleqn]{amsart}
\usepackage{amscd, amsfonts, amsthm, graphicx, amssymb}
\usepackage[all]{xy}
\title{Topological $\big(\prod^\omega \ell_2, \sum^\omega \ell_2\big)$-factors of diffeomorphism groups \\ 
of non-compact manifolds} 
\author{Tatsuhiko Yagasaki}
\address{Division of Mathematics,  Graduate School of Science and Technology, 
Kyoto Institute of Technology, 
Matsugasaki, Sakyoku, Kyoto 606-8585, Japan}
\email{yagasaki@kit.ac.jp}

\newtheorem{thm}{Theorem}[section]
\newtheorem{prop}{Proposition}[section] 
\newtheorem{cor}{Corollary}[section] 
\newtheorem{lemma}{Lemma}[section]

\theoremstyle{definition}
\newtheorem{defi}{Definition}[section]
\newtheorem{remark}{Remark}[section]

\newtheorem{example}{Example}[section]

\newtheorem*{claim}{Claim}

\newtheorem{assumption}{Assumption}[section] 
\newtheorem{compl}{Complement}[section] 
\newtheorem*{remark-notation}{Remark \& Notations}

\def \cal {\mathcal}
\def \phi {\varphi}
\def \ds {\displaystyle}

\def \lra {\longrightarrow}

\def \e {\varepsilon}
 
\def \id {{\rm id}} 
\def \ssp {\mbox{\tiny $\#$}}

\begin{document}
\baselineskip 6 mm

\thispagestyle{empty}

\maketitle


\begin{abstract} 
Suppose $M$ is a non-compact connected smooth $n$-manifold. 
Let ${\cal D}(M)$ denote the group of diffeomorphisms of $M$ 
endowed with the compact-open $C^\infty$-topology and 
${\cal D}^c(M)$ denote the subgroup consisting of diffeomorphisms of $M$ with compact support. 
Let ${\cal D}(M)_0$ and ${\cal D}^c(M)_0$ be the connected components of $\id_M$ in ${\cal D}(M)$ and ${\cal D}^c(M)$ respectively. 
In this paper we show that the pair $({\cal D}(M), {\cal D}^c(M))$ admits  
a topological $\big(\prod^\omega \ell_2, \sum^\omega \ell_2\big)$-factor. 
In the case $n = 2$, 
this enables us to apply the characterization of $\big(\prod^\omega \ell_2, \sum^\omega \ell_2\big)$-manifolds and 
show that the pair $({\cal D}(M)_0, {\cal D}^c(M)_0)$ is 
a $\big(\prod^\omega \ell_2, \sum^\omega \ell_2\big)$-manifold 
and determine its topological type. 
We also obtain a  similar result for groups of homeomorphisms of non-compact topological 2-manifolds. 
\end{abstract}


\section{Introduction}

This article is a continuation of study of topological properties of 
groups of homeomorphisms and diffeomorphisms of non-compact manifolds with the compact-open ($C^\infty$-)\,topology \cite{Ya2, Ya3, Ya4, Ya5}. 

Suppose $G$ is a transformation group acting on a space $M$ continuously and effectively. 
Each $g \in G$ induces a homeomorphism $\hat{g}$ of $M$. 
When $M$ is non-compact, the group $G$ contains the normal subgroup $G_c$  consisting of $g \in G$ such that $\hat{g}$ has a compact support. 
Let $G_0$ and $(G_c)_0$ denote the connected components of the unit element $e$ in $G$ and $G_c$ respectively. 
In this paper we are concerned with the topological type of the pair $(G_0, (G_c)_0)$  in the case where $G$ has a weak topology. 

Typical examples of the transformation groups $G$ are the group ${\cal H}(M)$ of homeomorphisms of  a topological manifold (or a locally compact polyhedron) $M$  
endowed with the compact-open topology and 
the group ${\cal D}(M)$ of diffeomorphisms of a smooth manifold $M$ 
endowed with the compact-open $C^\infty$-topology. 
In \cite{BMS} and \cite{BY} it is shown that both the pairs $({\cal H}(M)_0, {\cal H}^c(M)_0)$ for any countable infinite locally finite connected graph $M$ and $({\cal D}({\Bbb R})_0, {\cal D}^c({\Bbb R})_0)$ for the real line ${\Bbb R}$ are homeomorphic to the pair $\big(\prod^\omega \ell_2, \sum^\omega \ell_2\big)$. 
Here $\prod^\omega \ell_2$ is the countable product of the separable Hilbert space $\ell_2$ and $\sum^\omega \ell_2$ is the  countable weak product of $\ell_2$ defined by 
$$\mbox{$\sum^\omega \ell_2 = \big\{ (x_i)_i \in \prod^\omega \ell_2 \mid x_i = 0$ except finitely many $i\,\big\}$.}$$ 
In this paper, we show that the pairs $({\cal H}(M)_0, {\cal H}^c(M)_0)$ and $({\cal D}(M)_0, {\cal D}^c(M)_0)$ for a non-compact 2-manifold $M$ are $\big(\prod^\omega \ell_2, \sum^\omega \ell_2\big)$-manifolds (Theorem~\ref{thm_main}) and determine their topological types from their homotopy types (Corollary~\ref{cor_top-type}).  

To establish these results, first we deduce a characterization of 
$(\prod^\omega \ell_2, \sum^\omega \ell_2)$-manifolds 
under the stability property (Theorem~\ref{thm_(s,sf)}) 
from a general criterion \cite[Theorem 2.9]{Ya1}. 
A pair $(X, A)$ is said to be $(\prod^\omega \ell_2, \sum^\omega \ell_2)$-stable if $(X \times \prod^\omega \ell_2, A \times \sum^\omega \ell_2) \cong (X, A)$. Stability properties of homeomorphism groups of topological manifolds and their subgroups have already been studied by many authors (cf. \cite{Ge1, Ge2, Ke, KW, SW} etc). 
In particular, in \cite{Ya3} we have treated the non-compact case in detail. 
On the other hand, the Moser's theorem for volume forms \cite{Mos} (cf. \cite{Ya4})  exhibits the $\ell_2$-stability property of diffeomorphism groups. 
We modify these arguments and show the $(\prod^\omega \ell_2, \sum^\omega \ell_2)$-stability property of the pairs $({\cal H}(M), {\cal H}^c(M))$ and 
$({\cal D}(M), {\cal D}^c(M))$ for non-compact (separable metrizable) $n$-manifolds $M$. 

\begin{thm}\label{thm_stability} 
The pair $(G, G_c)$ is $\big(\prod^\omega \ell_2, \sum^\omega \ell_2\big)$-stable in the following cases: 
\begin{enumerate} 
\item $G = {\cal D}(M)$ for a non-compact smooth $n$-manifold $M$ 
possibly with boundary $(n \geq 1)$.
\item $G = {\cal H}(M)$ for a non-compact topological $n$-manifold $M$ 
possibly with boundary $(n \geq 1)$. 
\item $G = {\cal H}(M, \mu)$ for a non-compact topological $n$-manifold $M$ possibly with boundary $(n \geq 2)$ and a good Radon measure $\mu$ on $M$. 
\end{enumerate} 
\end{thm} 

See Section 5.2 for the group ${\cal H}(M, \mu)$ of $\mu$-preserving homeomorphisms of $M$.  

In dimension $n = 2$, combined with the ANR-property and homotopy density 
(\cite[Corollary 1.1]{Ya2}, \cite[Theorem 3.2, Corollary 3.1]{Ya3}, 
\cite[Theorems 1.1, 1.2]{Ya5}),    
this stability property 
enables us to apply the characterization of 
$(\prod^\omega \ell_2, \sum^\omega \ell_2)$-manifolds (Theorem~\ref{thm_(s,sf)}) to the pairs $({\cal D}(M)_0, {\cal D}^c(M)_0)$ and $({\cal H}(M)_0, {\cal H}^c(M)_0)$. 
For a subgroup $H$ of a transformation group $G$ on $M$ we use the following notations: 
$H_X = \{ g \in H \mid \hat{g}|_X = \id_X \}$ for $X \subset M$ and $H_c = H \cap G_c$. 
Let $H_1$ denote the path component of the unit element $e$ in $H$, and set 
$$(H_c)^\ast_1 = \cup \{ (H_{M - K})_1\mid \mbox{$K$ is a compact subset of $M$} \}.$$ 

\begin{thm}\label{thm_main} 
Suppose $G$ is one of the following groups: 
\begin{enumerate}
\item ${\cal D}_X(M)$ for a non-compact connected smooth 2-manifold $M$ and a compact submanifold $X$ of $M$. 
\item ${\cal H}_X(M)$ for a non-compact connected 2-manifold $M$ possibly with boundary and 
a compact subpolyhedron $X$ of $M$ with respect to some triangulation of $M$. 
\end{enumerate}
Then the pair $(G_0, H)$ is a $(\prod^\omega \ell_2, \sum^\omega \ell_2)$-manifold for any subgroup $H$ of $G_0$ such that 
$H$ is $F_\sigma$ in $G$ and 
$(G_c)_1^\ast \subset H \subset (G_0)_c$. 
\end{thm}

Note that the subgroups $H = (G_c)_1^\ast, (G_c)_0$ and $(G_0)_c$ 
satisfy the conditions in Theorem~\ref{thm_main}. 

The topological type of any $(\prod^\omega \ell_2, \sum^\omega \ell_2)$-manifold $(X, A)$ is classified by the homotopy type of $X$ (Theorem~\ref{thm_(s,sf)}). 
Hence, by \cite[Theorem 1.1]{Ya2} and \cite[Theorem 1.1]{Ya5} we have the conclusion on the global topological type. Consider the next two cases : 
\begin{enumerate}
\item[(I)\,] $(M, X) \cong ({\Bbb R}^2, \emptyset)$, $({\Bbb R}^2, 1pt)$, 
$({\Bbb S}^1 \times {\Bbb R}^1, \emptyset)$, $({\Bbb S}^1 \times [0, 1), \emptyset)$ or $({\Bbb P}^2 \setminus 1pt, \emptyset)$. 
\vskip 1mm 
\item[(II)] $(M, X)$ is not the case (I) (in the cases (1) and (2) in Theorem~\ref{thm_main}). 
\end{enumerate} 
Here ${\Bbb R}^n$ is the Euclidean $n$-space, ${\Bbb S}^n$ is the $n$-sphere and ${\Bbb P}^2$ is the projective plane.

\begin{cor}\label{cor_top-type} In Theorem~\ref{thm_main} we have  
$(G_0, H) \cong \begin{cases}
(\prod^\omega \ell_2, \sum^\omega \ell_2) \times {\Bbb S}^1 & \mbox{in the case $(I)$,} \\[1mm] 
(\prod^\omega \ell_2, \sum^\omega \ell_2) & \mbox{in the case $(I\!I)$.} 
\end{cases}$ 
\end{cor} 

This paper is organized as follows. 
In Section 2 we deduce the characterization of 
$(\prod^\omega \ell_2, \sum^\omega \ell_2)$-manifolds 
based upon the stability property (Theorem~\ref{thm_(s,sf)}). 
In Section 3 we obtain the results on the diffeomorphism groups in 
Theorems~\ref{thm_stability}, \ref{thm_main} and Corollary~\ref{cor_top-type}, while 
Section 4 includes the results on the homeomorphism groups.


\section{Characterization of topological $(\prod^\omega \ell_2, \sum^\omega \ell_2)$-manifolds}

In \cite{Ya1} we obtained a general characterization of 
infinite-dimensional manifold tuples based upon the stability property 
(cf. \cite{DT, vM, To}, \cite{BGM, BM, Ca, CDM}, \cite{GH}, etc.). 
In this section we deduce a characterization of $(\prod^\omega \ell_2, \sum^\omega \ell_2)$-manifolds from this general characterization theorem. 

\subsection{General characterization of infinite-dimensional manifold pairs under the stability property} \mbox{} 

We begin with the definition of basic terminology. 
In this paper spaces are assumed to be separable metrizable and maps are continuous. The symbol $\cong$ means a homeomorphism, 
while $\simeq$ means a homotopy equivalence. 
A pair of spaces means a pair $(X, A)$ of a topological space $X$ and a subset $A$ of $X$. 
We say that two pairs $(X, A)$ and $(Y, B)$ are homeomorphic and write $(X, A) \cong (Y, B)$ if there exists a homeomorphism $h : X \to Y$ with $h(A) = B$. 
For a model space $E$, an $E$-manifold means a space $X$ locally homeomorphic to $E$. More generally, 
for a model pair $(E, E_1)$, by an $(E, E_1)$-manifold we mean a pair $(X, X_1)$ of spaces such that 
each point $x$ of $X$ admits an open neighborhood $U$ of $x$ in $X$ and an open subset $V$ of $E$ such that $(U, U \cap X_1) \cong (V, V \cap E_1)$.

A closed subset $A$ of a space $X$ is called a $Z$-set of $X$ if for any open cover ${\cal U}$ of $X$ there exists a map $f : X \to X - A$ which is ${\cal U}$-close to $\id_X$.  
A $\sigma$ $Z$-set of $X$ means a countable union of $Z$-sets of $X$. 
A subset $A$ of $X$ is said to be homotopy dense (HD) if there exists a homotopy $h_t : X \to X$ $(t \in [0,1])$ such that $h_0 = \id_X$ and $h_t(X) \subset A$ for $t \in (0, 1]$. 

Consider the countable product $s = \prod_{k \in {\Bbb N}} {\Bbb R}$, which is a topological linear space under the coordinatewise sum and scalar product. 
Since $s$ is a separable Fr\'echet space, it follows that $s \cong \ell_2$. 
Suppose $s_1$ is a linear subspace of $s$. 
For $I \subset {\Bbb N}$ we set $c(I) = {\Bbb N} \setminus I$ and $s(I) = \prod_{k \in I} {\Bbb R}$, and let $\pi_I : s \to s(I)$ denote the projection. 
We set $s_1(I) = \pi_I(s_1) \subset s(I)$. 
Let ${\mathcal M} \equiv {\mathcal M}(s, s_1)$ denote the class of pairs $(X, A)$ which admit a closed embedding $h : X \to s$ such that $h^{-1}(s_1) = A$. 

\vspace{1mm} 

\begin{assumption}\label{ass_pair} We assume that the model pair $(s, s_1)$ satisfies the following conditions : 
\begin{itemize}
\item[($\ast_1$)] $s_1$ is a linear subspace of $s$ and $s_1$ is a $\sigma$ $Z$-set of $s_1$ itself. 
\item[($\ast_2$)] $s_1$ is homotopy dense in $s$. 
\item[($\ast_3$)] There exists a sequence $I_n \, (n \geq 1)$ of disjoint infinite subsets of ${\Bbb N}$ such that for each $n \geq 1$ (a) ${\rm min} \, I_n > n$, (b) $s_1 = s_1(I_n) \times s_1(c(I_n))$ and (c) $(s(I_n), s_1(I_n)) \cong (s, s_1)$.  
\end{itemize}
\end{assumption} 
\vspace{1mm} 

Under Assumption~\ref{ass_pair} we have the following characterization and homotopy invariance of $(s, s_1)$-manifolds. 
This is exactly the case that $\ell = 1$ in \cite[Theorem 2.9, Corollary 2.10]{Ya1})

\begin{thm}\label{thm_ch}
A pair $(X, A)$ is an $(s, s_1)$-manifold iff 
\begin{enumerate}
\item $X$ is a separable completely metrizable ANR,  
\item {\rm (i)} $(X, A) \in {\mathcal M}(s, s_1)$, \ {\rm (ii)} $A$ is homotopy dense in $X$, 
\item $(X, A)$ is $(s, s_1)$-stable. 
\end{enumerate}
\end{thm}

\begin{cor}\label{cor_hom-inv} 
Suppose $(X, A)$ and $(Y, B)$ are $(s, s_1)$-manifolds. 
Then $(X, A) \cong (Y, B)$ iff $X \simeq Y$. 
\end{cor} 

\subsection{Characterization of $(\prod^\omega \ell_2, \sum^\omega \ell_2)$-manifolds} \mbox{} 

Next we deduce a characterization and classification of $\big(\prod^\omega \ell_2, \sum^\omega \ell_2\big)$-manifolds from Theorem~\ref{thm_ch} and Corollary~\ref{cor_hom-inv}. 

\begin{thm}\label{thm_(s,sf)} \mbox{} 
\begin{itemize} 
\item[(1)] 
A pair $(X, A)$ is a $\big(\prod^\omega \ell_2, \sum^\omega \ell_2\big)$-manifold iff it satisfies the following conditions: 
\begin{itemize} 
\item[(i)\ ] $X$ is a separable completely metrizable ANR. 
\item[(ii)\,] {\rm (a)} $A$ is $F_\sigma$ in $X$, \ {\rm (b)} $A$ is homotopy dense in $X$. 
\item[(iii)] $(X, A)$ is $\big(\prod^\omega \ell_2, \sum^\omega \ell_2\big)$-stable. 
\end{itemize}
\item[(2)] Suppose $(X, A)$ and $(Y, B)$ are $\big(\prod^\omega \ell_2, \sum^\omega \ell_2\big)$-manifolds. 
Then $(X, A) \cong (Y, B)$ iff $X \simeq Y$. 
\end{itemize}
\end{thm}

Since $s \cong \ell_2$, it follows that $\big(\prod^\omega \ell_2, \sum^\omega \ell_2\big) \cong \big(\prod^\omega s, \sum^\omega s\big)$. 
The latter is also denoted by the symbol $(s^\infty, s^\infty_f)$ for notational simplicity. 
Since 
$s^\infty = \prod_{k \in {\Bbb N}} s = \prod_{n \in {\Bbb N}} \! \big(\prod_{k \in {\Bbb N}} {\Bbb R}\big) = \prod_{(n, k) \in {\Bbb N}^2} {\Bbb R}$, 
any bijection $\alpha : {\Bbb N}^2 \cong {\Bbb N}$ induces a linear homeomorphism $\phi_\alpha : s^\infty \cong s$ and a linear subspace $s_1 = \phi_\alpha(s^\infty_f)$ of $s$. 
Hence Theorem~\ref{thm_(s,sf)} follows from 
Theorem~\ref{thm_ch}, Corollary~\ref{cor_hom-inv} and the next two lemmas. 

\begin{lemma} The pair $(s, s_1)$ satisfies Assumption~\ref{ass_pair}. 
\end{lemma} 

\begin{proof} 
$(\ast_1)$, $(\ast_2)$ For each $n \geq 1$ the closed subset $s^n = s^n \times \{ (0, 0, \cdots) \}$ of $s^\infty$ satisfies the condition that 
$s^\infty_f - s^n$ is homotopy dense in $s^\infty$. 
In fact, with replacing the interval $[0,1]$ by $[n, \infty]$ in the opposite  orientation, 
an absorbing homotopy $\psi : s^\infty \times [n, \infty] \to s^\infty$ is defined by 
$$\mbox{
$\psi((x_i)_{i \in {\Bbb N}}, t) = 
\begin{cases}
(x_1, \cdots, x_k, (t - k)x_{k+1}, k+1 - t, t - k, 0, \cdots) &  
(t \in [k, k+1], k \geq n) \\[1mm]
(x_i)_{i \in {\Bbb N}} & (t = \infty).  
\end{cases}$ 
}$$ 
This implies that $s^\infty_f$ is homotopy dense in $s^\infty$ and 
that $s^n$ is a $Z$-set of $s^\infty_f$ for each $n \geq 1$. 
Since $s^\infty_f = \cup_{n = 1}^\infty s^n$, it follows that 
$s^\infty_f$ is a $\sigma$ $Z$-set of $s^\infty_f$ itself. 
Since $(s, s_1) \cong (s^\infty, s^\infty_f)$, this implies the conditions $(\ast_1)$ and $(\ast_2)$ for the pair $(s, s_1)$. 

$(\ast_3)$ 
For any infinite subset $J$ of ${\Bbb N}$ 
it is easily seen that the subset $J' = {\Bbb N} \times J$ of ${\Bbb N}^2$ satisfies the conditions: 
(b$'$) $s^\infty_f = s^\infty_f(J') \times s^\infty_f(c(J'))$ and (c$'$) $(s^\infty(J'), s^\infty_f(J')) \cong (s^\infty, s^\infty_f)$. 
Thus the subset $I = \alpha(J')$ of ${\Bbb N}$ satisfies 
the corresponding conditions: (b) $s_1 = s_1(I) \times s_1(c(I))$ and (c) $(s(I), s_1(I)) \cong (s, s_1)$. 
Inductively we can find a sequence $J_n$ $(n \geq 1)$ of disjoint infinite subsets  of ${\Bbb N}$ with $\min \alpha({J_n}') > n$.  
Then the subsets $I_n = \alpha({J_n}')$ $(n \geq 1)$ of ${\Bbb N}$ satisfy the required condition. 
\end{proof} 

\begin{lemma} $(X, A) \in {\mathcal M}(s^\infty, s^\infty_f)$ iff 
$X$ is separable completely metrizable and $A$ is $F_\sigma$ in $X$. 
\end{lemma} 

\begin{proof} Recall that $(X, A) \in {\mathcal M}(s^\infty, s^\infty_f)$ iff there exists a closed embedding $f : X \to s^\infty$ such that $f^{-1}(s^\infty_f) = A$. 
Since $s^\infty$ is separable completely metrizable and $s^\infty_f$ is $F_\sigma$ in $s^\infty$, 
any $(X, A) \in {\mathcal M}(s^\infty, s^\infty_f)$ satisfies the same conditions. 

Conversely, suppose $X$ is separable completely metrizable and $A$ is $F_\sigma$ in $X$. 
Then we can find a closed embedding $e : X \to s  = s^1 \subset s^\infty$ 
and a map $g : X \to s$ such that $g^{-1}(\sigma) = A$, 
where $\sigma = \sum^\omega {\Bbb R} = {\Bbb R}^\infty_f$. 
A suitable change of indices induces a homeomorphism of pairs 
$\chi : (s \times s^\infty, \sigma \times s^\infty_f) \cong (s^\infty, s^\infty_f)$. 
The required embedding $f : X \to s^\infty$ is defined by $f = \chi \circ (g, e)$. 
Indeed, (i) since $e$ is a closed embedding, so is $f$, and (ii) since $e(X) \subset s^1 \subset  s^\infty_f$, it follows that 
$$f^{-1}(s^\infty_f)  = (g, e)^{-1}\chi^{-1}(s^\infty_f) 
= (g, e)^{-1}(\sigma \times s^\infty_f) = A.$$
\vskip -7.5mm  
\end{proof} 


\section{Stability property of $(G, G_c)$-spaces}

To treat the groups of homeomorphisms and their subgroups systematically, 
we formulate our argument to transformation groups. 
If $E \cong F \times B$ and $B$ is $\ell_2$-stable, then $E$ itself is $\ell_2$-stable. Thus the study of stability property is reduced 
to seeking for infinite-dimensional factors.  

\subsection{Factorization of $G$-spaces} \mbox{} 

In this subsection we give a simple criterion that a $G$-space admits a product decomposition. 
Suppose $E$ is a space and $G$ is a topological group 
which acts continuously on $E$ from the right. 
We seek a condition that $E$ factors 
to a product of a subspace $F$ of $E$ and a space $B$. 
Consider three maps $p : E \to B$, $f : E \to F$ and $g : B \to G$, which induce 
two maps 
$$\mbox{$\phi : E \to B \times F$; $\phi(x) = (p(x), f(x))$ \ and \ 
$\psi : B \times F \to E$; $\psi(b, y) = y \cdot g(b)$.}$$ 

\begin{lemma}~\label{lem_factorization}
The maps $\phi$ and $\psi$ are reciplocal homeomorphisms iff 
$$\mbox{$(\ast)$ \ $f(x) \cdot g(p(x)) = x$ \ $(\forall \,x \in E)$ \ \ and \ \ 
$p(y \cdot g(b)) = b$ \ $(\forall \,(b, y) \in B \times F)$.}$$
\end{lemma}

\begin{proof} 
From the definition of the maps $\phi$ and $\psi$, we have the next identities: 
$$\begin{array}[t]{l}
\psi \phi(x) = \psi(p(x), f(x)) = f(x) \cdot g(p(x)). \\[2mm]  
\phi\psi(b, y) = \phi(y \cdot g(b)) = (p(y \cdot g(b)), f(y \cdot g(b))).
\end{array}$$ 
The condition $(\ast)$ implies that $\psi \phi(x) = x$ and 
$\phi\psi(b, y) = (b, f(y \cdot g(b))) = (b, y)$, since 
$$f(y \cdot g(b)) = (y \cdot g(b)) \cdot g(p(y \cdot g(b)))^{-1} = (y \cdot g(b)) \cdot g(b)^{-1} = y.$$ 
This means that $\psi = \phi^{-1}$. 
The converse is obvious. 
\end{proof} 

\begin{compl}\label{complement_factorization}
In addition, if (a) the maps $p$, $f$ and $g$ are maps of pairs 
$$\mbox{$p : (E, E_1) \to (B, B_1)$, $f : (E, E_1) \to (F, F_1)$ \ and \ $g : (B, B_1) \to (G, G_1)$,}$$ 
\begin{itemize}
\item[(b)] $F_1 \subset E_1$, and (c) 
$G_1$ is a subgroup of $G$ such that $E_1$ is $G_1$-invariant (i.e., $E_1 \cdot G_1 \subset E_1$), 
\end{itemize} 
then the maps $\phi$ and $\psi$ induce the  maps of pairs 
$$\mbox{$\phi : (E, E_1) \to (B, B_1) \times (F, F_1)$ \ \ and \ \ 
$\psi : (B, B_1) \times (F, F_1) \to (E, E_1)$.}$$  
By Lemma ~\ref{lem_factorization}, if the maps $p$, $f$ and $g$ satisfy the condition $(\ast)$, then 
the maps $\phi$ and $\psi$ are reciplocal homeomorphisms of pairs 
(and $F_1 = F \cap E_1$). 
\end{compl}  

The next lemma is the simplest case of Complement~\ref{complement_factorization}. 

\begin{lemma}\label{lem_special_case} 
Suppose the maps $p : E \to B$, $f : E \to F$ and $g : B \to G$ satisfy the condition $(\ast)$, so that 
the map $\phi : E \to B \times F$ is a homeomorphism. 
If $E'$ is a $G$-invariant subspace of $E$, then $\phi(E') = B \times (E' \cap F)$. 
\end{lemma} 

\begin{proof} In Complement~\ref{complement_factorization} we can take $(E_1, B_1, F_1, G_1) = (E', B, E' \cap F, G)$. 
For the condition $f(E_1) \subset F_1$ note that 
$f(x) = x \cdot g(p(x))^{-1} \in E_1 \cdot G = E_1$ for $x \in E_1$. 
\end{proof} 

\subsection{Transformation groups on non-compact spaces} \mbox{}

A transformation group on a space $M$ is a topological group $G$ which acts on $M$ continuously and faithfully.  
Each $g \in G$ induces a homeomorphism $\widehat{g}$ of $M$. 
For a subset $H$ of $G$ and a subset $K$ of $M$, let $H_K = \{ h \in H \mid \widehat{h} = \id \text{ on } K\}$ and $H(K) = H_{M - K}$. 

A support function for a space ${\cal E}$ on $M$ is a function which assigns to each $f \in {\cal E}$ a closed subset ${\rm supp}\,f$ of $M$. 
When a space ${\cal E}$ is equipped with a support function on $M$, 
for any subspace ${\cal F}$ of ${\cal E}$ we obtain the subspace 
${\cal F}^c = \{ f \in {\cal F} \mid {\rm supp}\,f \text{ is compact} \}$. 
For the transformation group $G$ on $M$ 
the support of $g \in G$ is canonically defined by 
${\rm supp}\,g = {\rm supp}\,\widehat{g} \equiv cl_M\{ x \in M \mid \widehat{g}(x) \neq x \}.$ 
In this case $G^c$ is a normal subgroup of $G$. 

\begin{defi}
We say that 
\begin{itemize}
\item[$(\ast_1)$] $G$ has a weak topology if for each neighborhood $U$ of $e$ in $G$ there exists a compact subset $K$ of $M$ such that $G_K \subset U$, 
\item[$(\ast_2)$] 
$G$ has the multiplication supported by 
a discrete family $\{ E_i \}_{i \in \Lambda}$ of compact subsets in $M$
if for any $(g_i)_i \in \prod_{i \in \Lambda} G(E_i)$ 
there exists a $g \in G(\cup_i E_i)$ such that 
$\widehat{g} = \widehat{\,g_i\,}$ on $E_i$ for each $i \in \Lambda$. 
The element $g$ is denoted by 
$\prod_{i \in \Lambda} g_i$.  
\end{itemize} 
\end{defi} 

\begin{remark} Suppose $G$ has the multiplication supported by 
a discrete family $\{ E_i \}_{i \in \Lambda}$ of compact subsets in $M$. 

(1) Since the action of $G$ on $M$ is faithful, each $g \in G$ is uniquely determined by $\widehat{g}$. Thus the element $\prod_{i \in \Lambda} g_i$ is uniquely determined by the defining property. 

(2) The multiplication map \ 
$$\mbox{
$\eta : \prod_{i \in \Lambda} G(E_i) \to G(\cup_i E_i)$; \ 
$\eta((g_i)_i) = \prod_{i \in \Lambda} g_i$ \ 
}$$ 
is a group homomorphism and 
$\eta^{-1}(G^c(\cup_i E_i)) = \sum_{i \in \Lambda} G(E_i)$. 

(3) Any disjoint partition $\Lambda = \Lambda_0 \cup \Lambda_1$ yields 
the product decomposition 
$\prod_{i \in \Lambda} G(E_i) = \prod_{i \in \Lambda_0} G(E_i) \times \prod_{i \in \Lambda_1} G(E_i)$, by which 
the group $\prod_{i \in \Lambda_0} G(E_i)$ is regarded as 
a subgroup of $\prod_{i \in \Lambda} G(E_i)$. 
Thus, for any $(g_i)_{i \in \Lambda_0} \in \prod_{i \in \Lambda_0} G(E_i)$ 
we obtain the product 
$\prod_{i \in \Lambda_0} g_i \in G(\cup_{i \in \Lambda_0} E_i)$.
When $\Lambda_0$ is a finite subset, 
the element $\prod_{i \in \Lambda_0} g_i$ 
coincides with the usual product of $g_i$'s in $G$, which is independent of the order of $g_i$'s. 
\end{remark}

\begin{lemma}\label{lem_comp} 
If $G$ has a weak topology and 
has the multiplication supported by 
a discrete family $\{ E_i \}_{i \in \Lambda}$ of compact subsets in $M$, then 
the multiplication map $\eta$ is continuous. 
\end{lemma} 

\begin{proof} Since $\eta$ is a group homomorphism between topological groups, 
it suffices to show that the map $\eta$ is continuous at the unit element $e_\Lambda = (e)_i$ of 
$\prod_i G(E_i)$. 
Given any neighborhood $U$ of $e$ in $G$, 
there exists a neighborhood $V$ of $e$ in $G$ and 
a compact subset $K$ of $M$ such that 
$V^2 \subset U$ and $G_K \subset V$. 
Since $\{ E_i \}_i$ is discrete, there exists a finite subset $\Lambda_0$ of $\Lambda$ such that $E_i \cap K = \emptyset$ for any $i \in \Lambda_1 \equiv \Lambda - \Lambda_0$. 
This partition induces the decomposition $\prod_{i \in \Lambda} G(E_i) = \prod_{i \in \Lambda_0} G(E_i) \times \prod_{i \in \Lambda_1} G(E_i)$. 

Since the finite multiplication map 
$$\mbox{$\eta_0 : \prod_{i \in \Lambda_0} G(E_i) \to G$; \ $\eta_0((g_i)_i) = \prod_{i \in \Lambda_0}g_i$}$$ 
is continuous, there exists a neighborhood $W_0$ of the unit element $e_{\Lambda_0}$ in $\prod_{i \in \Lambda_0} G(E_i)$ such that $\eta_0(W_0) \subset V$. 
Then, $W = W_0 \times \prod_{i \in \Lambda_1} G(E_i)$ is a neighborhood of $e_\Lambda$ in $\prod_{i \in \Lambda} G(E_i)$ and 
for any $g_\Lambda = (g_{\Lambda_0}, g_{\Lambda_1}) \in W$ it follows that 
(a) $g_{\Lambda_0} \in W_0$ and $\eta(g_{\Lambda_0}) = 
\eta_0(g_{\Lambda_0}) \in V$, 
(b) $\eta(g_{\Lambda_1}) \in G(\cup_{i \in \Lambda_1} E_i) \subset G_K \subset V$, so that 
$\eta(g_\Lambda) = \eta(g_{\Lambda_0}) \eta(g_{\Lambda_1}) \in V^2 \subset U$. 
This completes the proof. 
\end{proof} 


\subsection{Factorization of $(G, G^c)$-pairs} \mbox{}

In this subsection we incorporate the arguments in the previous subsections and deduce a practical criterion, Proposition~\ref{prop_stability}, which is used in Sections 4 and 5. Now consider the following data: 

\begin{assumption}\label{assump_stability} \mbox{} 
\begin{itemize}
\item[(1)] $M$ is a space, 
$\{ E_i \}_{i \in \Lambda}$ is a discrete family of compact subsets in $M$ 
and $D_i$ is a compact subset of $E_i$ for each $i \in \Lambda$.

\item[(2)] ${\cal G}$ is a transformation group on $M$ which has a weak topology and has the multiplication supported by the family $\{ E_i \}_{i \in \Lambda}$. 

\item[(3)] $({\cal E}, f_0)$ is a pointed space equipped with a support function on $M$ and a continuous  right action of ${\cal G}$. Suppose that 
\begin{itemize} 
\item[(i)] ${\rm supp}\,f_0 = \emptyset$ \ and \ (ii)\, ${\rm supp}\,fg \subset {\rm supp}\,f \cup {\rm supp}\,g$ \ for any $(f, g) \in {\cal E} \times {\cal G}$.
\end{itemize} 

\item[(4)] For each $i \in \Lambda$ 
(a) $(B_i, \alpha_i)$ is a pointed space and 
\begin{itemize}
\item[(b)] $P_i : ({\cal E}, f_0) \to (B_i, \alpha_i)$ and $G_i : (B_i, \alpha_i) \to ({\cal G}(E_i), e)$ are pointed maps. 
\end{itemize} 
Assume that these maps satisfy the following conditions. 
\begin{itemize} 
\item[(i)\ ] $P_i(f) = \alpha_i$ if $({\rm supp}\,f) \cap D_i = \emptyset$, 
\item[(ii)\,] $P_i(fg^{-1}) = \alpha_i$ for any $(f, g) \in {\cal E} \times {\cal G}$ with $\widehat{g} = \widehat{G_i(P_i(f))}$ on $D_i$. 
\item[(iii)] $P_i(fg) = \nu$ for any $(f, g, \nu) \in {\cal E} \times {\cal G} \times B_i$ with $P_i(f) = \alpha_i$ and $\widehat{g} = \widehat{G_i(\nu)}$ on $D_i$. 
\end{itemize} 
\end{itemize} 
\end{assumption}

Assumption 3.1 yields the following conclusions; 
By (2) and Lemma~\ref{lem_comp} the multiplication map 
$\eta : \prod_{i \in \Lambda} {\cal G}(E_i) \to {\cal G}(\cup_i E_i)$ is continuous. 
By (3)(ii) the subspace ${\cal E}^c$ is ${\cal G}^c$-invariant (i.e., ${\cal E}^c \cdot {\cal G}^c = {\cal E}^c$). 
For simplicity, we use the symbol 
$$\mbox{$({\cal B}, {\cal B}^c, \alpha) = \big(\prod_{i \in \Lambda} B_i, \sum_{i \in \Lambda} B_i, (\alpha_i)_i \big).$}$$ 
Recall that  
$\sum_{i \in \Lambda} B_i = \big\{ (\mu_i)_i \in \prod_{i \in \Lambda} B_i \mid 
\mu_i = \alpha_i \mbox{ except finitely many } i \in \Lambda \big\}$.
The maps $P_i$, $G_i$ $(i \in \Lambda)$ in (4) are combined to yield the following  maps between pointed pairs: 
\begin{itemize} 
\item[(5)] 
$P : ({\cal E}, {\cal E}^c, f_0) \to ({\cal B}, {\cal B}^c, \alpha)$, \ 
$G : ({\cal B}, {\cal B}^c, \alpha) \to ({\cal G}, {\cal G}^c, e)$ \ and \ 
$F : ({\cal E}, {\cal E}^c, f_0) \to ({\cal F}, {\cal F}^c, f_0)$. 
\end{itemize} 
These maps are defined by the formula: 
$$\begin{array}[t]{c}
P(f) = (P_i(f))_{i \in \Lambda} \ \ (f \in {\cal E}), \hspace{8mm} 
G(\mu) = \eta\big((G_i(\mu_i))_i\big) \ \ (\mu = (\mu_i)_i \in {\cal B}), \\[2mm] 
{\cal F} = P^{-1}(\alpha) \subset {\cal E}, \hspace{6mm} 
F(f) = f \cdot G(P(f))^{-1} \ \ (f \in {\cal E}).
\end{array}$$  

If $f \in {\cal E}^c$, then the compact set ${\rm supp}\,f$ meets only finitely many $D_i$'s. Thus by (4)(i) $P_i(f) = \alpha_i$ except finitely many $i$'s and so  
$P(f) \in  {\cal B}^c$. 
Since $G_i(\alpha_i) = e$ and 
$\eta\big(\sum_{i \in \Lambda} {\cal G}(E_i)\big) \subset {\cal G}^c(\cup_i E_i)$, we have $G({\cal B}^c) \subset {\cal G}^c$. 
For each $i \in \Lambda$, since $\widehat{G(P(f))} = \widehat{G_i(P_i(f))}$ on $D_i$, the condition (4)(ii) implies that 
$$P_i(F(f)) = P_i\big(f \cdot G(P(f))^{-1}\big) = \alpha_i.$$ 
This means that $PF(f) = \alpha$ and $F(f) \in {\cal F}$. 
If $f \in {\cal E}^c$, then $G(P(f)) \in {\cal G}^c$ and 
$F(f) \in {\cal E}^c \cdot {\cal G}^c = {\cal E}^c$.  
This implies that $F(f) \in {\cal F}^c$. 

The maps $P$, $F$ and $G$ determine two maps 
\begin{itemize}
\item[(6)] $\Phi : ({\cal E}, {\cal E}^c) \lra ({\cal B}, {\cal B}^c) \times ({\cal F}, {\cal F}^c)$ \ \ and \ \ 
$\Psi : ({\cal B}, {\cal B}^c) \times ({\cal F}, {\cal F}^c) \lra ({\cal E}, {\cal E}^c)$. 
\end{itemize}
These maps are defined by \ $\Phi(f) = (P(f), F(f))$ \ and \ $\Psi(\mu, h) = h \cdot G(\mu)$.  

\begin{prop}\label{prop_stability} 
{\rm (i)} The maps $\Phi$ and $\Psi$ are reciplocal homeomorphisms of pairs.

{\rm (ii)} {\rm (a)} Suppose $({\cal E}_1, {\cal E}_2)$ is a subpair of $({\cal E}, {\cal E}^c)$ and  
${\cal E}_1$ is ${\cal G}$-invariant and ${\cal E}_2$ is ${\cal G}^c$-invariant respectively. 
Then the homeomorphism $\Phi$ restricts to the homeomorphism of the subpairs 
$$\Phi : ({\cal E}_1, {\cal E}_2) \lra ({\cal B}, {\cal B}^c) \times ({\cal E}_1 \cap {\cal F}, {\cal E}_2 \cap {\cal F}).$$ 

{\rm (b)} In particular, if ${\cal E}_1$ is a ${\cal G}$-invariant subspace of ${\cal E}$, then 
the homeomorphism $\Phi$ induces the homeomorphism of the subpairs 
$$\Phi : ({\cal E}_1, {\cal E}_1^c) \lra ({\cal B}, {\cal B}^c) \times ({\cal E}_1 \cap {\cal F}, {\cal E}_1^c \cap {\cal F}).$$ 

\end{prop} 

\begin{proof} 
(i) By Complement~\ref{complement_factorization} 
it suffices to verify the following conditions:   
$$\mbox{$(\ast_1)$ \ $F(f) \cdot G(P(f)) = f$ \ for any $f \in {\cal E}$, \ \ 
$(\ast_2)$ $P(h \circ G(\mu)) = \mu$ \ for any $(\mu, h) \in {\cal B} \times {\cal F}$.}$$

The condition $(\ast_1)$ follows from the definition of the map $F$. 
Since $P(h) = \alpha$ and $\widehat{G(\mu)} = \widehat{G_i(\mu_i)}$ on $D_i$, 
by (4)(iii) it follows that $P_i(h \circ G(\mu)) = \mu_i$. 
This implies the condition $(\ast_2)$. 

(ii) (a) By the conditions on $({\cal E}_1, {\cal E}_2)$ the map $F$ induces the map between subpairs, $F : ({\cal E}_1, {\cal E}_2) \to  ({\cal E}_1 \cap {\cal F}, {\cal E}_2 \cap {\cal F})$. 
Thus the assertion follows from (i) and Complement~\ref{complement_factorization}.  

(b) Since ${\cal E}_1^c$ is ${\cal G}^c$-invariant, the statement follows from (a). 
\end{proof}


\section{Stability property of diffeomorphism groups} 

In this section we study the stability proeprty of diffeomorphism groups (Theorem~\ref{thm_stability}\,(1)) 
and prove Theorem~\ref{thm_main}\,(1). 

\subsection{Preliminaries on volume forms and volume densities} \mbox{}

Suppose $M$ is a smooth (separable metrizable) $n$-manifold possibly with boundary. 
When $M$ is orientable, the volume forms on $M$ serves our purpose. 
However, to include the non-orientable case, 
it is necessary to recall the notion of volume density.

Suppose $V$ is a 1-dimensional real vector space. 
The dual space $V^\ast$ consists of 
all linear functions $f : V \to {\Bbb R}$, while  
its variant $V^{\ssp}$ is defined by 
$$V^{\ssp} = \{ f \in V \to {\Bbb R} \mid f(av) = |a| f(v) \ (\forall v \in V, \forall a \in {\Bbb R})\}
= \{ \pm |f| \mid f \in V^\ast \}.$$ 
These spaces form 1-dimensional real vector spaces under the usual sum and scalor product of real-valued functions. 
When $V$ is oriented, by $V_+$ we denote 
the connected component of $V - \{ 0 \}$ consisting of positive vectors. 
Even if $V$ itself is not oriented, 
the space $V^{\ssp}$ always admits a canonical orientation with the positive vectors 
$V^{\ssp}_+ = \{ |f| \mid f \in V^\ast \}$. 
In addition, if $V$ is oriented, then $V^\ast$ admits the corresponding orientation and  a canonical orientation-preserving isomorphism $V^\ast \cong V^{\ssp}$. 

This construction extends to real line bundles. 
For any smooth real line bundle $L \to M$,  
we have the associated real line bundles $L^\ast = \cup_{p \in M} (E_p)^\ast \to M$ and 
$L^{\ssp} = \cup_{p \in M} (L_p)^{\ssp} \to M$. 
The line bundle $L^{\ssp}$ has a canonical orientation, so that 
it is trivial since $M$ is paracompact. 
If $L$ itself is oriented 
(i.e., each fiber $L_p$ $(p \in M)$ is equipped with an orientation $o_p$ which varies continuously in $p \in M$), 
then $L^\ast$ also admits the corresponding orientation and 
there exists a canonical isomorphism $L^\ast \cong L^{\ssp}$ of oriented vector bundles over $M$. 

For a smooth manifold $N$ possibly with boundary, 
let $C^\infty(M, N)$ denote the space of $C^\infty$ maps $f : M \to N$.
More generally, for a smooth fiber bundle $E \to M$, 
let $\Gamma(E)$ denote the space of global sections of $E$.  
These spaces are endowed with the compact-open $C^\infty$-topology. 
Note that $\Gamma(M \times N) \cong C^\infty(M, N)$ for the product bundle $M \times N \to M$. 
For an oriented real line bundle $L \to M$, the subspace of positive sections of $L$ is defined by 
$$\Gamma_+(L) = \{ s \in \Gamma(L) \mid s(p) \in (L_p)_+ \ (\forall p \in M) \}.$$ 

\begin{lemma}\label{lem_ell_2} 
$\Gamma_+(M; L) \cong \ell_2$. 
\end{lemma} 

\begin{proof} 
The trivial fiber bundle $L$ includes the sub-bundle $L_+ = \cup_{p \in M} (L_p)_+$, which is also a trivial fiber bundle with fiber $(0, \infty) \cong {\Bbb R}$. 
Since $\Gamma_+(L) = \Gamma(L_+) \cong C^\infty(M, {\Bbb R})$ and 
the latter is an infinite-dimensional separable Frechet space, we have the conclusion. 
\end{proof} 

Now we apply the above arguments to the line bundle $\wedge^n T M$.  
Any section $\omega \in \Gamma((\wedge^n T M)^{\ssp})$ is called a density on $M$, 
\vspace{1mm} 
since its components over coordinate charts transform by 
the absolute value of Jacobian under coordinate transitions and hence 
the integral $\ds \int_M \omega \in {\Bbb R}$ is well-defined whenever $\omega$ has a compact support and 
\vspace{1mm} 
the $\omega$-volume $\ds \omega(M) = \int_M \omega \in (0, \infty]$ is defined  as an improper integral for any positive density $\omega \in \Gamma_+((\wedge^n T M)^{\ssp})$. 
To simplify the notations, let 
$$\mbox{${\cal V}^{\ssp}_{(+)}(M) = \Gamma_{(+)}((\wedge^n T M)^{\ssp})$ \ \  and \ \ 
${\cal V}^{\ssp}_+(M; m) = \{ \omega \in {\cal V}^{\ssp}_+(M) \mid \omega(M) = m \}$ \ $(m \in (0, \infty])$. }$$ 

Suppose $N$ is another smooth $n$-manifold possibly with boundary and 
$f : N \to M$ is a $C^\infty$ map. 
Then the differential of $f$, $df : TN \to TM$, induces the pull-back 
$f^\ast\omega \in {\cal V}^{\ssp}(N)$ for each 
$\omega \in {\cal V}^{\ssp}(M)$. 
This defines a continuous map 
$$f^\ast : {\cal V}^{\ssp}(M) \to {\cal V}^{\ssp}(N).$$
Moreover, if $f$ is an immersion, then $f^\ast({\cal V}^{\ssp}_+(M)) \subset {\cal V}^{\ssp}_+(N)$. 
It is seen that the group ${\cal D}(M)$ acts continuously on the space ${\cal V}^{\ssp}(M)$ by the pull-back and the subspace ${\cal V}^{\ssp}_+(M)$ is invariant under this action. 
For the inclusion $i : N \subset M$, the pull-back $i^\ast \omega$ is also denoted by $\omega|_N$. 

For $\mu, \nu \in {\cal V}^{\ssp}(M)$ we write $\mu \sim_1 \nu$ if $\nu = c \mu$ for some $c > 0$. 
This is an equivalence relation and preserved by the pull-back. 

When $M$ is oriented  (i.e., $\wedge^n T M$ is oriented), 
the line bundle $(\wedge^n T M)^\ast$ also admits a canonical orientation and 
there exist canonical isomorphisms of oriented vector bundles 
$$(\wedge^n T M)^{\ssp} \cong (\wedge^n T M)^\ast \cong \wedge^n T^\ast M.$$  
This enables us to identify the space 
${\cal V}^{\ssp}_+(M)$  
with the space ${\cal V}_+(M) = \Gamma_+(\wedge^n T^\ast M)$ 
of positive volume forms on $M$. 
The induced homeomorphism $\eta_M :{\cal V}^{\ssp}_+(M) \cong {\cal V}_+(M)$ is compatible with the pull-back, that is, 
for any orientation-preserving diffeomorphism $h \in {\cal D}_+(M)$ we have the commutative diagram: 
$$\xymatrix@M+3pt{
{\cal V}^{\ssp}_+(M) \ar[d]^{\cong}_{\mbox{$\eta_M$}} 
\ar[r]_{\cong}^{\mbox{$h^\ast$}} &  
{\cal V}^{\ssp}_+(M) \ar[d]_{\cong}^{\mbox{$\eta_M$}} \\
{\cal V}_+(M) \ar[r]_{\cong}^{\mbox{$h^\ast$}} & {\cal V}_+(M)
}$$ 

Suppose $F$ is a compact smooth $n$-manifold possibly with boundary. 
\vspace{1mm} 
For $\mu \in {\cal V}^{\ssp}_+(F)$ we define $\ds \overline{\mu} = \frac{1}{\mu(F)}\mu \in {\cal V}^{\ssp}_+(F; 1)$. Then, $\mu \sim_1 \nu$ iff $\overline{\mu} = \overline{\nu}$ for $\mu, \nu \in {\cal V}^{\ssp}_+(F)$.

\begin{lemma}\label{lem_convex}
${\cal V}^{\ssp}_+(F; 1) \cong \ell_2$.  
\end{lemma} 

\begin{proof} The space ${\cal V}^{\ssp}_+(F; 1)$ is a convex subset of the separable Fr\'echet space ${\cal V}^{\ssp}(F) = \Gamma((\wedge^n T F)^{\ssp})$, and 
it admits a canonical homeomorphism  
$$\chi : {\cal V}^{\ssp}_+(F) \cong {\cal V}^{\ssp}_+(F; 1) \times (0, \infty), \ \ \chi(\mu) = (\overline{\mu}, \mu(F)).$$ 
The inverse is given by $\chi^{-1}(\mu, c) = c \mu$. 
Since ${\cal V}^{\ssp}_+(F) = \Gamma_+((\wedge^n T F)^{\ssp}) \cong \ell_2$ by Lemma~\ref{lem_ell_2}, 
it follows that ${\cal V}^{\ssp}_+(F; 1)$ is nowhere locally compact. 
Hence, ${\cal V}^{\ssp}_+(F; 1) \cong \ell_2$ by \cite{DT}.  
\end{proof} 

In the next subsection we need a version of Moser's theorem for volume forms  (\cite[Theorem 2]{Mos} cf. \cite[Theorem 2.1]{Ya4}). 
Suppose $M$ is a smooth $n$-manifold possibly with boundary, 
$L \subset N$ are compact connected oriented smooth $n$-submanifolds of $M$ possibly with boundary such that $L \subset {\rm Int}\,N$ and $L$ inherits the orientation from $N$, 
and $\alpha \in {\cal V}^{\ssp}_+(L; 1)$. 
Consider the following subgroup of ${\cal D}(M)$: 
$${\cal G}(N, L) = \{ h \in {\cal D}_{M - {\rm Int}\,N}(M) \mid h(L) = L\}.$$ 
For any subgroup ${\cal H}$ of ${\cal D}(M)$, 
we define a subgroup ${\cal H}_1^{\natural}$ of ${\cal H}$ by  
$${\cal H}_1^{\natural} = \{ h \in {\cal H} \mid h : (\natural)_h \}: \hspace{5mm} 
(\natural)_h : \begin{array}[t]{l} 
\mbox{There exists a smooth isotopy $H : M \times [0,1] \to M$} \\[1.5mm] 
\mbox{such that $H_0 = \id_M, H_1 = h$ \ and \ $H_t \in {\cal H}$ \ $(\forall \,t \in [0,1])$.}  
\end{array}$$ 

\begin{thm}\label{thm_Moser} 
There exists a map $\phi : {\cal V}^{\ssp}_+(L) \to {\cal G}(N, L)_1^{\natural}$ such that \\[1mm]  
\hspace*{50mm} $(\phi(\mu)|_L)^\ast\alpha \sim_1 \mu$ \ $(\forall \, \mu \in {\cal V}^{\ssp}_+(L))$ \ and \ $\phi(\alpha) = \id_M$. 
\end{thm} 

\begin{proof} 
Under the canonical homeomorphism 
$\eta_L : {\cal V}^{\ssp}_+(L) \cong {\cal V}_+(L)$, 
the map $\phi$ corresponds with the map 
$\phi' : {\cal V}_+(L) \to {\cal G}(N, L)_1^{\natural}$ such that \\[1mm]  
\hspace*{50mm} $(\phi'(\mu)|_L)^\ast\alpha' \sim_1 \mu$ \ $(\forall \, \mu \in {\cal V}_+(L))$ \ and \ $\phi'(\alpha') = \id_M$, \\
where $\alpha' = \eta_L(\alpha) \in {\cal V}_+(L; 1)$. 
Below we construct the map $\phi'$. 

There exists a map $s : {\cal V}_+(L) \to {\cal V}_+(N)$ such that 
\begin{itemize}
\item[(i)\ ] $s(\mu)|_L = \mu$ \ $(\forall \,\mu \in {\cal V}_+(L))$. 
\end{itemize} 
Let $\beta = s(\alpha') \in {\cal V}_+(N)$. 
Choose a small bicollar $E = \partial L \times [-1, 1]$ of $\partial L$ in ${\rm Int}\,N$ with $\partial L \times [-1, 0] \subset L$. 
Then, by the parametrized version of \cite[Lemma A2]{Mc} (cf. \cite[Lemma 2.3]{Ya4}) we can find maps 
$$\mbox{$\psi : {\cal V}_+(N) \to {\cal D}_{(M - {\rm Int}\,E) \cup \partial L}(M)_1^{\natural} \subset {\cal G}(N, L)_1^{\natural}$ \ \ and \ \ $\e : {\cal V}_+(N) \to (0, 1/2)$}$$ 
such that 
\begin{itemize}
\item[(ii)\,] $\ds (\psi(\nu)|_N)^\ast\beta = \frac{1}{\nu(L)}\nu$ \ on \ $\partial L \times [-2\e(\nu), 2\e(\nu)]$ \ $(\forall \,\nu \in {\cal V}_+(N))$ \ \ and \ \ (ii)$'$ \ $\psi(\beta) = \id_M$. 
\end{itemize} 
\vskip 1mm 
Consider two maps $\lambda, \kappa : {\cal V}_+(N) \to {\cal V}_+(L; 1)$ defined by 
$$\mbox{$\lambda(\nu) = ((\psi(\nu)|_N)^\ast\beta)|_L = (\psi(\nu)|_L)^\ast\alpha'$ \ \ and \ \ 
$\kappa(\nu) = \overline{\nu|_L}$.}$$  
Then, we have 
\begin{itemize}
\item[(iii)] $\ds \lambda(\nu) = (\psi(\nu)|_N)^\ast\beta 
= \frac{1}{\nu(L)}\nu = \kappa(\nu)$ \ on \ $\partial L \times [-2\e(\nu), 0]$ \ \ and \ \ (iii)$'$ \ $\lambda(\beta) = \kappa(\beta) = \alpha'$. 
\end{itemize} 
Thus, the parametrized version of Moser's Theorem \cite[Theorem 2]{Mos} (cf. \cite[Theorem 2.1]{Ya4}) yields a map 
$$\chi : {\cal V}_+(N) \to {\cal D}_{M - {\rm Int}\,L}(M)_1^{\natural} \subset {\cal G}({N, L})_1^{\natural}$$ 
such that 
\begin{itemize}
\item[(iv)] $(\chi(\nu)|_L)^\ast \lambda(\nu) = \kappa(\nu)$ \ $(\forall\,\nu \in {\cal V}_+(N))$ \ \ and \ \ (iv)$'$ \ $\chi(\beta) = \id_M$. 
\end{itemize} 

The required map $\phi' : {\cal V}_+(L) \to {\cal G}(N, L)_1^{\natural}$ is defined by 
$$\phi'(\mu) = \psi(s(\mu))\chi(s(\mu)).$$ 
Indeed, one sees that,  
\begin{itemize}
\item[(v)\,] 
$\begin{array}[t]{@{}l@{}l@{}ll}
(\phi'(\mu)|_L)^\ast\alpha' & \ =\ & (\psi(s(\mu)|_L\,\chi(s(\mu))|_L)^\ast\alpha' 
= (\chi(s(\mu))|_L)^\ast (\psi(s(\mu)|_L)^\ast \alpha' & \\[1.5mm]  
& \ =\ & (\chi(s(\mu)|_L)^\ast\lambda(s(\mu)) 
= \kappa(s(\mu))
= \overline{\mu} & \text{$(\forall\,\mu \in {\cal V}_+(L))$ \ and}
\end{array}$ 
\vskip 1.5mm 
\item[(v)$'$] $\phi'(\alpha') = \psi(\beta)\chi(\beta) = \id_M$. 
\end{itemize}
This completes the proof. 
\end{proof} 


\subsection{$\big(\prod_{i \in \Lambda} \ell_2, \sum_{i \in \Lambda} \ell_2 \big)$-stability of diffeomorphism groups} \mbox{} 

Suppose $M$ and $N$ are smooth $n$-manifolds possibly with boundary. 
We fix a section $\omega \in {\cal V}^{\ssp}_+(N)$. 
Let $\{ E_i \}_{i \in \Lambda}$ be a discrete family of oriented smooth closed $n$-disks in $M$. 
Since $M$ is assumed to be separable, the index set $\Lambda$ is at most countable and  
when $M$ is non-compact, we can take $\Lambda$ to be an infinite countable set. 
For each $i \in \Lambda$ take an $n$-subdisk $D_i$ in ${\rm Int}\,E_i$, 
which inherits the orientation from $E_i$. 

Consider the subspace ${\cal E}$ of $C^\infty(M, N)$ defined by  
\[\hspace{10mm} \mbox{${\cal E} = \big\{ f \in C^\infty(M, N) \mid 
f|_{D_i} : D_i \to N$ is a $C^\infty$-immersion for each $i \in \Lambda\big\}$.}\]
Below we study the stability property of the space ${\cal E}$ and its subspaces,  based upon the arguments in Section 3.3. 

Define the subgroup ${\cal G}$ of ${\cal D}(M)$ by \ ${\cal G} = ({\cal G}')_1^{\natural}$, where 
$$\mbox{${\cal G}' = \big\{ h \in {\cal D}(M) \mid 
h = \id_M \text{ on } M - \cup_{i \in \Lambda} {\rm Int}\,E_i, \  
h(D_i) = D_i \text{ for each } i \in \Lambda \big\}$.}$$  
The group ${\cal G}$ acts continuously on the space ${\cal E}$ by the right composition, and as a transformation group on $M$ it 
has a weak topology and also 
has the multiplication supported by the family $\{ E_i \}_{i \in \Lambda}$. 
Hence, the composition map 
$\eta : \prod_{i \in \Lambda} {\cal G}(E_i) \to {\cal G}$ is continuous by Lemma~\ref{lem_comp}.  

Since ${\cal E} \neq \emptyset$, we can choose a distinguished element 
$f_0 \in {\cal E}$. 
A support function for ${\cal E}$ on $M$ is defined by 
$${\rm supp}_{f_0}f = cl_M\{ x \in M \mid f(x) \neq f_0(x) \}.$$  
Note that it satisfies the condition in Assumption~\ref{assump_stability}\,(3) and the subspace ${\cal E}^c$ is ${\cal G}^c$-invariant (i.e., ${\cal E}^c \cdot {\cal G}^c = {\cal E}^c$). 

For each $i \in \Lambda$ define a pointed space $(B_i, \alpha_i)$ and two maps $P_i$ and $G_i$ as follows: Let 
\begin{itemize} 
\item[] $(B_i, \alpha_i) = \big({\cal V}^{\ssp}_+(D_i; 1), \overline{(f_0|_{D_i})^\ast \omega}\big)$. 
\end{itemize} 
Theorem~\ref{thm_Moser} yields a map 
$\phi_i : {\cal V}^{\ssp}_+(D_i) \to {\cal G}$ such that 
$$\mbox{(i) $(\phi_i(\lambda)|_{D_i})^\ast\alpha_i \sim_1 \lambda$ \ $(\forall \,  \lambda \in {\cal V}^{\ssp}_+(D_i))$ \ \ and \ \ (ii) $\phi_i(\alpha_i) = \id_M$.}$$  
The maps $P_i$ and $G_i$ are defined by 
\begin{itemize} 
\item[] $P_i : ({\cal E}, f_0) \to (B_i, \alpha_i); \ \ P_i(f) = \overline{(f|_{D_i})^\ast \omega}$, 
\vskip 1mm 
\item[] $G_i : (B_i, \alpha_i) \to ({\cal G}, \id_M); \ \ 
G_i(\lambda) = \phi_i(\lambda)$.   
\end{itemize} 

\begin{claim} The maps $P_i$ and $G_i$ satisfy the conditions (i) - (iii) in Assumption~\ref{assump_stability}\,(4). 
\end{claim}

\begin{proof} 
(i) If $({\rm supp}_{f_0}f) \cap D_i = \emptyset$, then 
$P_i(f) = \overline{(f|_{D_i})^\ast \omega} = \overline{(f_0|_{D_i})^\ast \omega} = \alpha_i$. 

(ii) Suppose $(f, g) \in {\cal E} \times {\cal G}$ and $g = G_i(P_i(f))$ on $D_i$. 
For $\lambda = P_i(f) \, (= \overline{(f|_{D_i})^\ast \omega})$, 
it is seen that  \  
$$\mbox{$(f|_{D_i})^\ast\omega \sim_1 \lambda$ \ \ and \ \ 
$(\phi_i(\lambda)|_{D_i})^\ast \alpha_i \sim_1 \lambda$\,, \ so that}$$ 
$$\big((\phi_i(\lambda)|_{D_i})^{-1}\big)^\ast (f|_{D_i})^\ast\omega \sim_1 \big((\phi_i(\lambda)|_{D_i})^{-1}\big)^\ast \lambda \sim_1 \alpha_i.$$ 
\vskip 1mm 
\noindent Since \ 
$(fg^{-1})|_{D_i} = f|_{D_i} (g|_{D_i})^{-1} = f|_{D_i}(G_i(\lambda)|_{D_i})^{-1}
= f|_{D_i}(\phi_i(\lambda)|_{D_i})^{-1}$\,, 
it follows that \  
$$P_i(fg^{-1}) 
\ = \ \overline{((fg^{-1})|_{D_i})^\ast \omega \makebox(0,9){}}
\ = \ \overline{(f|_{D_i}(\phi_i(\lambda)|_{D_i})^{-1})^\ast \omega \makebox(0,9){}}
\ = \ \overline{((\phi_i(\lambda)|_{D_i})^{-1})^\ast(f|_{D_i})^\ast \omega \makebox(0,9){}}
\ = \ \alpha_i.$$

(iii) Suppose $(f, g, \lambda) \in {\cal E} \times {\cal G} \times B_i$, 
$P_i(f) = \alpha_i$ and $g = G_i(\lambda)$ on $D_i$. 
Then one has 
$$\mbox{$P_i(f) = \alpha_i$, \ \ $(\phi_i(\lambda)|_{D_i})^\ast\alpha_i \sim_1 \lambda$ \ \ 
 and \ \ 
$(fg)|_{D_i} = f|_{D_i}g|_{D_i} = f|_{D_i}G_i(\lambda)|_{D_i}
= f|_{D_i}\phi_i(\lambda)|_{D_i}.$}$$ 
This implies that $(f|_{D_i})^\ast \omega \sim_1 \alpha_i$ and  
$$\hspace{8mm} 
((fg)|_{D_i}\big)^\ast \omega = (f|_{D_i}\phi_i(\lambda)|_{D_i})^\ast \omega
= (\phi_i(\lambda)|_{D_i})^\ast(f|_{D_i})^\ast \omega 
\sim_1 (\phi_i(\lambda)|_{D_i})^\ast \alpha_i 
\sim_1 \lambda,$$ 
so that $P_i(fg) = \overline{((fg)|_{D_i})^\ast \omega} = \lambda$. 
\end{proof} 

Hence, we can apply the arguments in Section 3.3 to this setting. 
Two pointed pairs $({\cal B}, {\cal B}^c, \alpha)$ and $({\cal F}, {\cal F}^c, f_0)$ and three maps $P$, $G$ and $F$ are defined by 
\begin{itemize}
\item[] $({\cal B}, {\cal B}^c, \alpha) = \big(\prod_{i \in \Lambda} B_i, \sum_{i \in \Lambda} B_i, (\alpha_i)_i \big)$, 
\vskip 1mm 
\item[] $P : ({\cal E}, {\cal E}^c, f_0) \to ({\cal B}, {\cal B}^c, \alpha)$; \ \ 
$P(f) = (P_i(f))_{i \in \Lambda}$ \ $(f \in {\cal E})$,  
\vskip 1mm 
\item[] $G : ({\cal B}, {\cal B}^c, \alpha) \to ({\cal G}, {\cal G}^c, \id_M)$; \ \ 
$G(\mu) = \eta\big((G_i(\mu_i))_i\big)$ \ $(\mu = (\mu_i)_i \in {\cal B})$,
\vskip 1mm 
\item[] ${\cal F} = P^{-1}(\alpha)$ \ \ $({\cal F}^c = {\cal F} \cap {\cal E}^c$), 
\vskip 1mm 
\item[] $F : ({\cal E}, {\cal E}^c, f_0) \to ({\cal F}, {\cal F}^c, f_0)$; \ \ 
$F(f) = f \cdot G(P(f))^{-1}$ \ $(f \in {\cal E})$.
\end{itemize}

These maps 
determine two maps $\Phi$ and $\Psi$ by  
\begin{itemize}
\item[] $\Phi : ({\cal E}, {\cal E}^c, f_0) \lra ({\cal B}, {\cal B}^c, \alpha) \times ({\cal F}, {\cal F}^c, f_0)$; \ \ $\Phi(f) = (P(f), F(f))$, 
\vskip 1mm 
\item[] $\Psi : ({\cal B}, {\cal B}^c, \alpha) \times ({\cal F}, {\cal F}^c, f_0) \lra ({\cal E}, {\cal E}^c, f_0)$; \ \ $\Psi(\mu, g) = g \cdot G(\mu)$. 
\end{itemize} 

\begin{lemma}\label{lem_diff_stability} 
$({\cal B}, {\cal B}^c) \cong \big(\prod_{i \in \Lambda} \ell_2, \sum_{i \in \Lambda} \ell_2 \big)$. 
Thus, the pair $({\cal B}, {\cal B}^c)$ is $\big(\prod_{i \in \Lambda} \ell_2, \sum_{i \in \Lambda} \ell_2 \big)$-stable. 
\end{lemma}

\begin{proof} 
From Lemma~\ref{lem_convex} it follows that $B_i = {\cal V}^{\ssp}_+(D_i; 1) \cong \ell_2$ for each $i \in \Lambda$ and $({\cal B}, {\cal B}^c) \cong\big(\prod_{i \in \Lambda} \ell_2, \sum_{i \in \Lambda} \ell_2 \big)$. 
Note that the pair 
$\big(\prod_{i \in \Lambda} \ell_2, \sum_{i \in \Lambda} \ell_2 \big)$ itself is 
$\big(\prod_{i \in \Lambda} \ell_2, \sum_{i \in \Lambda} \ell_2 \big)$-stable, since 
$$\mbox{$(\ell_2)^2 \cong \ell_2$ \ \ and  \ \ $\big(\prod_{i \in \Lambda} \ell_2, \sum_{i \in \Lambda} \ell_2\big)^2 
\cong 
\big(\prod_{i \in \Lambda} (\ell_2)^2, \sum_{i \in \Lambda} (\ell_2)^2\big)
\cong 
\big(\prod_{i \in \Lambda} \ell_2, \sum_{i \in \Lambda} \ell_2\big).$}$$
Thus the pair $({\cal B}, {\cal B}^c)$ is also $\big(\prod_{i \in \Lambda} \ell_2, \sum_{i \in \Lambda} \ell_2 \big)$-stable. 
\end{proof} 

The next proposition follows from Proposition~\ref{prop_stability}. 

\begin{prop}\label{prop_diff_stability} 
{\rm (1)} The maps $\Phi$ and $\Psi$ are reciplocal homeomorphisms of pairs.

{\rm (2)} {\rm (i)} If $({\cal E}_1, {\cal E}_2)$ is a subpair of $({\cal E}, {\cal E}^c)$,  
${\cal E}_1$ is ${\cal G}$-invariant and ${\cal E}_2$ is ${\cal G}^c$-invariant, then  
the map $\Phi$ restricts to the homeomorphism of the subpairs 
$$\Phi : ({\cal E}_1, {\cal E}_2) \lra ({\cal B}, {\cal B}^c) \times ({\cal E}_1 \cap {\cal F}, {\cal E}_2 \cap {\cal F}).$$ 
In particular, the pair $ ({\cal E}_1, {\cal E}_2)$ is $\big(\prod_{i \in \Lambda} \ell_2, \sum_{i \in \Lambda} \ell_2 \big)$-stable. 

{\rm (ii)} If ${\cal E}_1$ is a ${\cal G}$-invariant subspace of ${\cal E}$, then 
the map $\Phi$ induces the homeomorphism of the subpairs 
$$\Phi : ({\cal E}_1, {\cal E}_1^c) \lra ({\cal B}, {\cal B}^c) \times ( {\cal F}_1, {\cal F}_1^c) \hspace{5mm} \mbox{where \ ${\cal F}_1 = {\cal E}_1 \cap {\cal F}$}.$$ 
\end{prop}

\begin{example}
The space ${\cal E}$ includes the following ${\cal G}$-invariant subspaces;  
\begin{itemize}
\item[] ${\cal E}_1 :={\rm Imm}^\infty(M, N) = \{ f \in C^\infty(M, N) \mid f$ is a $C^\infty$-immersion$\}$, 
\item[] ${\cal E}_2 :={\rm Emb}^\infty(M, N) = \{ f \in C^\infty(M, N) \mid f$ is a $C^\infty$-embedding$\}$, 
\item[] ${\cal E}_3 :={\rm Cov}^\infty(M, N) = \{ f \in C^\infty(M, N) \mid f$ is a $C^\infty$-covering  projection$\}$. 
\end{itemize}
For each $i = 1, 2, 3$, the map $\Phi$ induces the homeomorphism between the subpairs 
$$\Phi : ({\cal E}_i, {\cal E}_i^c) \cong 
({\cal B}, {\cal B}^c) \times ({\cal F}_i, {\cal F}_i^c), \hspace{10mm} 
\text{where ${\cal F}_i = {\cal E}_i \cap {\cal F}$}.$$ 
Thus, the pair $({\cal E}_i, {\cal E}_i^c)$ is $\big(\prod_{i \in \Lambda} \ell_2, \sum_{i \in \Lambda} \ell_2 \big)$-stable. 
\end{example}

Next we consider the case where $M = N$. 
As a base point of the space ${\cal E}$ we take $f_0 = \id_M$. 
Then the support function ${\rm supp}_{f_0}$  
reduces to the ordinary support function. 
The space ${\cal E}$ includes the group ${\cal D}(M)$  
as a ${\cal G}$-invariant subspace. 
Below we discuss the $\big(\prod_{i \in \Lambda} \ell_2, \sum_{i \in \Lambda} \ell_2 \big)$-stability of the pair $({\cal D}(M), {\cal D}^c(M))$ and 
its subpairs. 
For any subset ${\cal H}$ of ${\cal D}(M)$ with $\id_M \in {\cal H}$, 
the symbols ${\cal H}_0$ and ${\cal H}_1$ denote 
the connected component and the path component of $\id_M$ in ${\cal H}$ respectively.   

From Proposition~\ref{prop_diff_stability}\,(2) we have the following criterion. 

\begin{prop}\label{prop_diff-gr_stability} Suppose $({\cal H}, {\cal K})$ is a pair of subgroups of ${\cal D}(M)$ with 
$({\cal G}, {\cal G}^c) \subset ({\cal H}, {\cal K}) \subset ({\cal D}(M), {\cal D}^c(M))$.
Then the map $\Phi$ induces the homeomorphism 
$$\Phi : ({\cal H}, {\cal K}) \cong 
({\cal B}, {\cal B}^c) \times ({\cal H} \cap {\cal F}, {\cal K} \cap {\cal F}).$$ 
Thus, the pair $({\cal H}, {\cal K})$ is $\big(\prod_{i \in \Lambda} \ell_2, \sum_{i \in \Lambda} \ell_2 \big)$-stable. 
\end{prop} 

The next example includes Theorem~\ref{thm_stability}\,(1). 

\begin{example} 
The pairs $({\cal D}_X(M), {\cal D}^c_X(M))$ and $({\cal D}_X(M)_i, {\cal D}_X^c(M)_i)$ $(i=0,1)$ are $\big(\prod_{i \in \Lambda} \ell_2, \sum_{i \in \Lambda} \ell_2 \big)$-stable for any subset $X$ of $M - \cup_{i\in\Lambda}{\rm Int}\,E_i$.  
Indeed, by the definition of ${\cal G}$ we have 
${\cal G} \subset {\cal D}_X(M)_1$ and 
${\cal G}^c \subset {\cal D}_X^c(M)_1^\ast \subset {\cal D}_X^c(M)_1$.  
Hence, the conclusion follows from Proposition~\ref{prop_diff-gr_stability}. 
\end{example} 

Now we have come to the position to apply 
the characterization of $(\prod^\omega \ell_2, \sum^\omega \ell_2)$-manifolds (Theorem~\ref{thm_(s,sf)}) based upon the stability (Proposition~\ref{prop_diff-gr_stability}). 

\begin{proof}[\bf Proof of Theorem~\ref{thm_main}\,(1)] It suffices to prove that the pair $(G_0, H)$ satisfies the conditions (i) $\sim$ (iii) in Theorem~\ref{thm_(s,sf)}\,(1). 
\begin{itemize}
\item[(i)\ ] By \cite[Theorem 1.1]{Ya5} 
the group $G_0$ is a separable completely metrizable ANR. 

\item[(ii)\,]  
(a) By the assumption $H$ is $F_\sigma$ in $G_0$.  
(b) Since $(G_c)_1^\ast \subset H$ and $(G_c)_1^\ast$ is homotopy dense in $G_0$ 
by \cite[Theorem 1.2]{Ya5}, the subgroup $H$ is also homotopy dense in $G_0$. 

\item[(iii)]
Since ${\cal G}^c \subset (G_c)_1^\ast$ and so $({\cal G}, {\cal G}^c) \subset (G_0, H) \subset ({\cal D}(M), {\cal D}^c(M))$, 
the $(\prod^\omega \ell_2, \sum^\omega \ell_2)$-stability of $(G_0, H)$ follows from Proposition~\ref{prop_diff-gr_stability}.  
\end{itemize}
This completes the proof.  
\end{proof}


\section{Stability property of homeomorphism groups} 

Stability properties of homeomorphism groups and their subgroups have already been studied by many authors \cite{Ge1, Ge2, Ke, KW, SW, Ya3}. 
K.\,Sakai and R.Y.\,Wong \cite{SW} showed that for Euclidean polyhedra $X$ and $Y$ the triples of homeomorphism groups and spaces of embeddings: 
$$\mbox{$({\cal H}(X), {\cal H}^{{\rm LIP}}(X), {\cal H}^{\rm PL}(X))$ \ \ and \ \ 
$({\rm Emb}\,(X, Y), {\rm Emb}^{{\rm LIP}}(X, Y), {\rm Emb}^{\rm PL}(X, Y))$}$$ 
are $(s, \Sigma, \sigma)$-stable. 
In \cite[Section 3]{Ya3} we discussed 
the case that $X$ is a non-compact polyhedron and showed that, for instance, 
the tuple
$$({\mathcal H}(X), {\mathcal H}^{\rm loc\,LIP}(X), {\mathcal H}^{\rm LIP}(X), {\mathcal H}^{\rm LIP, c}(X), 
{\mathcal H}^{\rm PL}(X), {\mathcal H}^{\rm PL, c}(X))$$ 
is $(s^{\infty}, \Sigma^{\infty}, s^\infty_b, \Sigma^{\infty}_f, \sigma^{\infty}, \sigma^{\infty}_f)$-stable. 
These arguments are based upon the Morse's $\mu$-length of arcs. 

In this section we retrace these arguments due to the formulation in Section 3.3 
and show that the pairs 
$({\cal H}(X), {\cal H}^c(X))$ and $({\cal H}(X; \mu), {\cal H}^c(X; \mu))$ are 
stable with respect to the pair $(s^{\infty}, s^{\infty}_f) \cong (\prod^\omega \ell_2, \sum^\omega \ell_2)$ (Theorem~\ref{thm_stability}\,(2)(3)).
Comparing with \cite{SW, Ya3}, here our emphasis is put on measure-preserving homeomorphisms. 

The following notations are used below; $C(X, Y)$ is 
the space of continuous maps $f : X \to Y$ endowed with the compact-open topology, ${\cal E}(X, Y)$ denotes the subspace of $C(X, Y)$ 
consisting of topological embeddings.  
For a subset $A$ of $X$ let 
$${\cal C \cal E}(X, A; Y) = \big\{ f \in C(X, Y) \mid f|_A : A \to Y \text{ is a topological embedding}\big\}$$ and 
${\cal H}(X, A) = \{ h \in {\cal H}(X) \mid h(A) = A\}$. 
For any subset ${\cal F}$ of ${\cal H}(X)$ with $\id_X \in {\cal F}$, 
the symbols ${\cal F}_0$ and ${\cal F}_1$ denote 
the connected component and the path component of $\id_X$ in ${\cal F}$ respectively.   
We regard as $s = (-1, 1)^\infty$ instead of ${\Bbb R}^\infty$ if necessary, and  use the symbol $(s^\Lambda, s^\Lambda_f)$ to denote the pair
$\big(\prod_{i \in \Lambda} s, \sum_{i \in \Lambda} s \big) \cong \big(\prod_{i \in \Lambda} \ell_2, \sum_{i \in \Lambda} \ell_2 \big)$ for notational simplicity. 

\subsection{Morse's $\mu$-length of arcs} \mbox{} 

Suppose $(X, d)$ is a metric space and $A$ is an arc in $X$. 
The arc $A$ admits a canonical linear order $\leq$ unique up to the reversion. 
For each $k \geq 1$ set $$S_k = \{ {\boldsymbol a} = (a_0, a_1, \cdots, a_k) \in A^{k+1} \mid a_0 \leq a_1 \leq \cdots \leq a_k\},$$ 
$$\mbox{$\delta_k({\boldsymbol a}) = \min\{d(a_{i-1}, a_i) \mid i = 1, \cdots, k\}$ \ (${\boldsymbol a} \in S_k$) \ \ and \ \  
$\mu_k(A) = \sup\{\delta_k({\boldsymbol a}) \mid {\boldsymbol a} \in S_k\}$.}$$ 
The $\mu$-length of $A$ is defined by $$\mu(A) = \sum_{k=1}^\infty 2^{-k} \mu_k(A).$$ 
We use the following property of the quantity $\mu(A)$. 

\begin{lemma}\label{lem_mu-length} {\rm  (\cite[\S1.~pp.~197--202]{SW}, cf. \cite[Lemma 4.3]{Ya3})} 
\begin{itemize} 
\item[(i)\,] For each $f \in {\cal E}([-1,1], X)$ there is a unique $t_f \in (-1, 1)$ such that 
$\mu(f([-1, t_f])) = \mu(f([t_f, 1]))$.
\item[(ii)] The function \ $\gamma : {\cal E}([-1,1], X) \to (-1,1)$, $f \mapsto t_f$, \ is continuous. 
\end{itemize}
\end{lemma} 


\subsection{Selection theorem for good Radon measures} \mbox{} 

Next we recall some basic facts on good Radon measures. 
A Radon measure on a space $X$ is a 
Borel measure $\mu$ on $X$ such that 
$\mu(K) < \infty$ for any compact subset $K$ of $X$.
The measure $\mu$ is called good if $\mu(p) = 0$ for any point $p \in X$ and 
$\mu(U) > 0$ for any nonempty open subset $U$ of $X$. 
Let ${\cal M}(X)$ denote the space of all Radon measures $\nu$ on $X$ endowed with the weak topology. 
For $\mu, \nu \in {\cal M}(X)$ 
we say that $\nu$ is $\mu$-biregular if 
$\mu(A) = 0$ iff $\nu(A) = 0$ for any Borel subset $A$ of $X$. 
For $\mu \in {\cal M}(X)$ and a Borel subset $A$ of $X$ let 
\vspace{1mm} 
$$\begin{array}[c]{l}
{\cal M}^A_g(X) 
= \big\{ \nu \in {\cal M}(X) \mid \nu(A) = 0, \ \text{$\nu$ is good} \big\}
\ \ \text{and} \\[2mm] 
{\cal M}^A_g(X; \mu\mbox{-reg}) 
= \big\{ \nu \in {\cal M}_g^A(X) \mid \nu(A) = \mu(A), \text{ $\nu$ is $\mu$-biregular} \big\}.
\end{array}$$ 
We say that $h \in {\cal H}(X)$ is 
$\mu$-biregular provided   
$\mu(h(B)) = 0$ iff $\mu(B) = 0$ for any Borel subset $B$ of $X$. 
The group ${\cal H}(X)$ includes two subgroups 
\vspace{1mm} 
$$\begin{array}[c]{l}
{\cal H}(X; \mu\mbox{-reg}) = \{ h \in {\cal H}(X) \mid h \ \text{is $\mu$-biregular} \}
\ \ \text{and} \\[2mm] 
{\cal H}(X; \mu) = \{ h \in {\cal H}(X) \mid h_\ast \mu = \mu \ 
(\mbox{i.e., $h$ is $\mu$-preserving})\}.
\end{array}$$ 

We need Oxtoby-Ulam theorem (\cite{OU}) and Fathi's selection theorem (\cite{Fa}). 

\begin{thm}\label{thm_Radon_measure} 
Suppose $N$ is a compact connected $n$-manifold possibly with boundary 
and $\mu \in {\cal M}_g^\partial(N)$. 
\begin{itemize}
\item[(1)] For any $\nu \in {\cal M}_g^{\partial}(N)$ with $\nu(N) = \mu(N)$
there is $h \in {\cal H}_\partial(N)$ such that $h_\ast \mu = \nu$. 

\item[(2)] There exists a continuous map 
$$\sigma : ({\cal M}_g^\partial(N; \mu\mbox{-reg}), \mu) \to ({\cal H}_\partial(N; \mu\mbox{-reg})_1, \id_N)$$   
such that $\sigma(\nu)_\ast \mu = \nu$ 
for any $\nu \in {\cal M}_g^{\partial}(N; \mu\mbox{-reg})$. 
\end{itemize}
\end{thm}

A typical example of a good Radon measure is the Lubesgue measure $m$ on ${\Bbb R}^n$ $(n \geq 1)$. 
For any subpolyhedron $K$ of ${\Bbb R}^n$ it is seen that  
${\cal H}^{\rm PL}(K) \subset {\cal H}(K; m\mbox{-reg})$. 
Consider the $n$-cube ${\Bbb B}^n = [-1,1]^n$ in ${\Bbb R}^n$. 
We identify the interval $[-1, 1]$ with the segment in ${\Bbb B}^n$ connecting the two points $(\pm 1, 0,\cdots, 0)$. (This means that $t \in [-1, 1]$ represents the point $(t, 0, \cdots, 0) \in {\Bbb B}^n$.) 

Consider a triple $(E, A, \mu)$, where 
(a) $E$ is a closed $n$-disk and $A$ is an arc in $E$ such that 
$(E, A) \cong ({\Bbb B}^n, [-1,1])$ (i.e., $A$ is a unknoted proper arc in $E$) and 
(b) $\mu \in {\cal M}_g^{A \cup \partial E}(E)$ for $n \geq 2$ and it is an empty structure for $n = 1$ (so that the symbol $\mu$ can be eliminated from the notation). 

\begin{lemma}\label{lem_zeta} 
For any homeomorphism $\theta : [-1,1] \cong A$ (i.e., a parametrization of $A$ by $[-1,1]$) there exists a map 
$$\zeta : ((-1, 1), 0) \to ({\cal H}_\partial(E, A; \mu)_1, \id_E)$$ 
such that $\zeta(t)(\theta(t)) = \theta(0)$ for each $t \in (-1, 1)$. 
\end{lemma} 

\begin{proof} 
When $n = 1$,  
the assertion is obvious. Below we assume that $n \geq 2$. 

(1) First we treat the case where $(E, A) = ({\Bbb B}^n, [-1,1])$, $\theta = \id_{[-1,1]}$ 
and $\mu$ is $m$-biregular. 
Consider the decomposition of the $n$-cube, 
${\Bbb B}^n = B_+ \cup_{B_0} B_-$, where 
$$\mbox{$B_+ = \{ x \in {\Bbb B}^n \mid x_n \geq 0\}$, \ 
$B_- = \{ x \in {\Bbb B}^n \mid x_n \leq 0\}$, \ 
$B_0 = \{ x \in {\Bbb B}^n \mid x_n = 0\} 
= B_+ \cap B_1$.}$$ 
We can easily find a map 
$$\xi : ((-1, 1), 0) \to \big({\cal H}_\partial^{\rm PL}({\Bbb B}^n, [-1,1]), \id_{{\Bbb B}^n}\big)$$ 
such that $\xi(t)(t) = 0$ and 
$\xi(t)(B_\pm) = B_\pm$. 
For $\mu_\pm := \mu|_{B_\pm} \in {\cal M}_g^\partial(B_\pm)$, 
Theorem~\ref{thm_Radon_measure}\,(2) induces maps 
$$\sigma_\pm : \big({\cal M}_g^\partial(B_\pm; \mu_\pm\mbox{-reg}), \mu_\pm\big) \to \big({\cal H}_\partial(B_\pm; \mu_\pm\mbox{-reg})_1, \id_{B_\pm})$$   
such that $\sigma_\pm(\nu)_\ast \mu_\pm = \nu$ 
for any $\nu \in {\cal M}_g^{\partial}(B_\pm; \mu_\pm\mbox{-reg})$. 
Since 
$$\xi(t)|_{B_\pm} \in {\cal H}^{\rm PL}(B_\pm) 
\subset {\cal H}(B_\pm; m\mbox{-reg}) = {\cal H}(B_\pm; \mu_\pm\mbox{-reg})$$ 
we can define the map
$\zeta : (-1, 1) \to {\cal H}_\partial({\Bbb B}^n, [-1,1]; \mu)$ by 
$$\mbox{$\zeta(t)|_{B_\pm} 
= \big(\sigma_\pm((\xi(t)|_{B_\pm})_\ast \mu_\pm)\big)^{-1} \xi(t)|_{B_\pm}$.}$$ 

(2) To reduce the general case to the case (1),  
we construct a homeomorphism 
$\overline{\theta} : ({\Bbb B}^n, [-1,1]) \cong (E, A)$ such that 
$\overline{\theta}|_{[-1,1]} = \theta$ and the pull-back 
${\overline{\theta}\,}^\ast\!\mu$ is $m$-biregular. 
Since $A$ is unknoted in $E$, 
the homeomorphism $\theta : [-1,1] \cong A$ admits 
an extension 
$\overline{\theta}_1 : {\Bbb B}^n \cong E$. 
For $\mu_1 := {\overline{\theta}_1\!}^\ast \mu \in {\cal M}_g^{\partial}({\Bbb B}^n)$, since $\mu_1([-1,1]) = 0$, 
we can find $\overline{\theta}_2 \in {\cal H}_{[-1,1] \cup \partial}({\Bbb B}^n)$ such that $\mu_1(\overline{\theta}_2(B_0)) = 0$. 
Then $\mu_2 := {\overline{\theta}_2\!}^\ast \mu_1 \in {\cal M}_g^{B_0 \cup \partial}({\Bbb B}^n)$ restricts to $\mu_2|_{B_\pm} \in {\cal M}_g^{\partial}(B_\pm)$ and Theorem~\ref{thm_Radon_measure}\,(1) 
yields homeomorphisms 
${\overline{\theta}_3\!}^\pm \in {\cal H}_\partial(B_\pm)$ such that 
$({\overline{\theta}_3\!}^\pm)^\ast \mu_2|_{B_\pm} = c_\pm m|_{B_\pm}$, where $c_\pm = \mu_2(B_\pm)/m(B_\pm)$. 
Define $\overline{\theta}_3 \in {\cal H}_{B_0 \cup \partial}({\Bbb B}^n)$ by $\overline{\theta}_3|_{B_\pm} = {\overline{\theta}_3\!}^\pm$. 
Then ${\overline{\theta}_3\!}^\ast \mu_2$ is $m$-biregular and hence 
$\overline{\theta} = \overline{\theta}_1 \overline{\theta}_2 \overline{\theta}_3$ satisfies the required conditions. 

Since $\mu' := {\overline{\theta}\,}^\ast \!\mu$ is $m$-biregular, 
the case (1) yields a map 
$$\zeta' : ((-1, 1),0) \to \big({\cal H}_\partial({\Bbb B}^n, [-1,1]; \mu'), \id_{{\Bbb B}^n}\big)$$ 
such that $\zeta'(t)(t) = 0$ for any $t \in (-1, 1)$. 
The required map $\zeta$ is defined by 
$$\mbox{$\zeta(t) = \overline{\theta} \, \zeta'(t) \,{\overline{\theta}\,}^{-1}$ \ \ $(t \in (-1,1))$.}$$ 
\vskip -7mm 
\end{proof}

Suppose $X$ is a metric space. 

\begin{lemma}\label{lem_property_mu}
For any $f_0 \in {\cal C \cal E}(E, A; X)$ there exist maps 
$$\mbox{
$\phi : ({\cal C \cal E}(E, A; X), f_0) \to (s, 0)$ \ \ and \ \  
$\psi : (s, 0) \to ({\cal H}_\partial(E, A; \mu)_1, \id_E)$,}$$ 
which satisfy the following conditions:  
\begin{itemize}
\item[(i)\ ] $\phi(f) = \phi(f')$ for any $f, f' \in {\cal C \cal E}(E, A; X)$ with $f = f'$ on $A$. 
\vskip 0.5mm 
\item[(ii)\,] $\phi(f \,\mbox{\small $\circ$} \, \psi(\phi(f))^{-1}) = 0$ for any $f \in {\cal C \cal E}(E, A; X)$. 
\vskip 0.5mm 
\item[(iii)] $\phi(f\psi(t)) = t$ for any $(f, t) \in \phi^{-1}(0) \times s$. 
\end{itemize} 
\end{lemma} 

\begin{proof}
There exists a disjoint family of closed $n$-disks $D_k$ $(k \in {\Bbb N})$ in $E$ such that ${\rm diam}\, D_k \to 0$ $(k \to \infty)$, 
$\mu(\partial D_k) = 0$ and $(D_k, D_k \cap A) \cong ({\Bbb B}^n, [-1,1])$. 
Recall the map $\gamma$ in Lemma~\ref{lem_mu-length}. 
For each $k \in {\Bbb N}$ 
choose a homeomorphism $\theta_k : [-1,1] \cong D_k \cap A$ such that 
$\gamma(f_0\theta_k) = 0$. Then 
we can apply Lemma~\ref{lem_zeta} to the triple $(D_k, D_k \cap A, \mu|_{D_k})$ and $\theta_k$ to obtain the map  
$$\zeta_k : ((-1, 1), 0) \to ({\cal H}_\partial(D_k, D_k \cap A; \mu|_{D_k})_1, \id_{D_k})$$ 
such that $\zeta_k(t)(\theta_k(t)) = \theta_k(0)$ for each $t \in (-1, 1)$. 

Define the maps $\phi$ and $\psi$ by 
\begin{itemize} 
\item[(a)] $\phi(f) = (\gamma(f \theta_k))_{k \in {\Bbb N}}$ \ \ $(f \in {\cal C \cal E}(E, A; X))$, 
\item[(b)] $\psi(t)|_{D_k} = \zeta_k(t_k)$ \ $(k \in {\Bbb N})$ \ and \ $\psi(t)|_{E - \cup_k D_k} = \id$ \ \ $(t = (t_k)_k \in s)$.
\end{itemize} 
It remains to verify the properties (i) $\sim$ (iii). 
First note that for each $k \in {\Bbb N}$ 
$$\mbox{$f\psi(t)\theta_k 
= f \psi(t)|_{D_k} \theta_k
= f \zeta_k(t_k) \theta_k$ \ \ and \ \ 
$f\psi(t)^{-1}\theta_k 
= f(\psi(t)|_{D_k})^{-1}\theta_k
= f \zeta_k(t_k)^{-1}\theta_k$.}$$

\begin{itemize} 
\item[(i)\ ] Since $f \theta_k = f|_A \theta_k$, one sees that 
$\phi(f)$ depends only on $f|_A$. 
\vskip 0.5mm 

\item[(ii)\,] Let $t = \phi(f)$. Then we have \ \ 
$t_k = \gamma(f \theta_k)$ \ \ and \ \  
$\zeta_k(t_k)^{-1}\theta_k(0) = \theta_k(t_k)$. \ \ 
Thus, from the definition of the map $\gamma$ it follows that 
$$\mbox{$\gamma(f\psi(\phi(f))^{-1}\theta_k) 
= \gamma(f\psi(t)^{-1}\theta_k) 
= \gamma(f \zeta_k(t_k)^{-1}\theta_k) = 0$ \ \ and \ \ $\phi(f\psi(\phi(f))^{-1}) = 0$.}$$
\vskip 0.5mm 
\item[(iii)] Note that \ \ $\zeta_k(t_k)(\theta_k(t_k)) = \theta_k(0)$ \ \ and \ \ $\gamma(f \theta_k) = 0$ \ \ since $\phi(f) = 0$. Thus, it follows that 
$$\mbox{
$\gamma(f\psi(t)\theta_k) 
= \gamma(f \zeta_k(t_k) \theta_k) 
= t_k$ \ \ and \ \ 
$\phi(f\psi(t)) = t$.}$$ 
\end{itemize} 
\vskip -7mm 
\end{proof} 


\subsection{$(s^\infty, s^\infty_f)$-stability of homeomorphism groups} \mbox{} 

Suppose $M$ and $N$ are metric spaces and 
$\{ (E_i, A_i, \mu_i) \}_{i \in \Lambda}$ is a family of triples such that 
\begin{itemize}
\item[(i)\,] $\{ E_i \}_{i \in \Lambda}$ is discrete family of topological closed disks  in $M$, 
\item[(ii)] for each $i \in \Lambda$, 
\begin{itemize}
\item[(a)] $n_i := \dim\,E_i \geq 1$ and ${\rm Int}\,E_i$ is open in $M$, 
\item[(b)] $A_i$ is an arc in $E_i$ such that 
$(E_i, A_i) \cong ({\Bbb B}^{n_i}, [-1,1])$, and 
\item[(c)] if $n_i \geq 2$, then $\mu_i \in {\cal M}_g^{A_i \cup \partial}(E_i)$  and if $n_i = 1$, then $\mu_i$ is an empty structure 
(so it can be eliminated from the notations). 
\end{itemize}
\end{itemize}

Consider the subspace ${\cal E}$ of $C(M, N)$ and the subgroup ${\cal G}$ of ${\cal H}(M)$ defined by  
$$\begin{array}[c]{l}
{\cal E} = \big\{ f \in C(M, N) \mid 
f|_{A_i} : A_i \to N$ is an embedding for each $i \in \Lambda\big\}, \\[1.5mm] 
{\cal G} = \big\{ h \in {\cal H}(M) \mid 
h = \id_M \text{ on } M - \cup_{i \in \Lambda} {\rm Int}\,E_i, \  
h|_{E_i} \in {\cal H}_\partial(E_i, A_i; \mu_i)_1 \text{ for each } i \in \Lambda \big\}.
\end{array}$$ 
The group ${\cal G}$ acts continuously on the space ${\cal E}$ by the right composition. Moreover, ${\cal G}$ is path connected. In fact, the pair 
$({\cal G}, {\cal G}^c)$ is isomorphic to $\big(\prod, \sum\big)_{i \in \Lambda} {\cal H}_\partial(E_i, A_i; \mu_i)_1$ as a pair of a topological group and its subgroup. 
As a transformation group on $M$, the group ${\cal G}$ has a weak topology and 
has the multiplication supported by the family $\{ E_i \}_{i \in \Lambda}$. 
Hence, the multiplication map 
$\eta : \prod_{i \in \Lambda} {\cal G}(E_i) \to {\cal G}$ is continuous.  

We assume that $N$ includes an arc, so that ${\cal E} \neq \emptyset$ and 
we can fix a distinguished element $f_0 \in {\cal E}$. 
The associated support function for ${\cal E}$ on $M$ is defined by 
$${\rm supp}_{f_0}f = cl_M\{ x \in M \mid f(x) \neq f_0(x) \} \hspace{5mm} (f \in {\cal E}).$$  
It satisfies the conditions (i), (ii) in Assumption~\ref{assump_stability}\,(3) and the subspace ${\cal E}^c$ is ${\cal G}^c$-invariant. 

For each triple $(E_i, A_i, \mu_i)$ and $f_0|_{E_i} \in {\cal C \cal E}(E_i, A_i; N)$,  Lemma~\ref{lem_property_mu} yields two maps 
$$\mbox{
$\phi_i : ({\cal C \cal E}(E_i, A_i; N), f_0|_{E_i}) \to (s, 0)$ \ \ and \ \  
$\psi_i : (s, 0) \to ({\cal H}_\partial(E_i, A_i; \mu_i)_1, \id_{E_i})$,}$$ 
such that 
\begin{itemize}
\item[(i)\ ] $\phi_i(f) = \phi_i(f')$ for any $f, f' \in {\cal C \cal E}(E_i, A_i; N)$ with $f = f'$ on $A_i$. 
\vskip 0.5mm 
\item[(ii)\,] $\phi_i(f \, \mbox{\small $\circ$} \, \psi_i(\phi_i(f))^{-1}) = 0$ for any $f \in {\cal C \cal E}(E_i, A_i; N)$. 
\vskip 0.5mm 
\item[(iii)] $\phi_i(f\psi_i(t)) = t$ for any $(f, t) \in \phi_i^{-1}(0) \times s$. 
\end{itemize} 

We define two maps $P_i$ and $G_i$ by 
$$\mbox{$P_i : ({\cal E}, f_0) \to (s, 0); \ P_i(f) = \phi_i(f|_{E_i})$ \ \ and  \ \ 
$G_i : (s, 0) \to ({\cal G}(E_i), \id_M); \ G_i(t)|_{E_i} = \psi_i(t)$.}$$ 
The next claim follows directly from the properties of the maps $\phi_i$ and $\psi_i$.  

\begin{claim} The maps $P_i$ and $G_i$ satisfy the following conditions. 
\begin{itemize} 
\item[(0)\,] $P_i(f) = P_i(f')$ if $f, f' \in {\cal E}$ and $f = f'$ on $A_i$. 
\item[(i)\ ] $P_i(f) = 0$ if $f \in {\cal E}$ and $({\rm supp}_{f_0}f) \cap A_i = \emptyset$. 
\item[(ii)\,] $P_i(fg^{-1}) = 0$ if $(f, g) \in {\cal E} \times {\cal G}$ and $g = G_i(P_i(f))$ on $A_i$. 
\item[(iii)] $P_i(fg) = t$ if $(f, g, t) \in {\cal E} \times {\cal G} \times s$, $P_i(f) = 0$ and $g = G_i(t)$ on $A_i$. 
\end{itemize} 
\end{claim}

This claim means that the maps $P_i$ and $G_i$ 
satisfy the conditions (i) - (iii) in Assumption~\ref{assump_stability}\,(4). 
Hence, we can apply the arguments in Section 3.3 to this situation. 
The pointed pairs $({\cal B}, {\cal B}^c, \alpha)$ and $({\cal F}, {\cal F}^c, f_0)$ and three maps $P$, $G$ and $F$ are defined by 
\begin{itemize}
\item[] $({\cal B}, {\cal B}^c, \alpha) = (s^\Lambda, s^\Lambda_f, 0) = \big(\prod_{i \in \Lambda}s, \sum_{i \in \Lambda} s, 0\big)$,  
\vskip 1mm 
\item[] $P : ({\cal E}, {\cal E}^c, f_0) \to (s^\Lambda, s^\Lambda_f, 0)$; \ \ 
$P(f) = (P_i(f))_{i \in \Lambda}$ \ $(f \in {\cal E})$,  
\vskip 1mm 
\item[] $G : (s^\Lambda, s^\Lambda_f, 0) \to ({\cal G}, {\cal G}^c, \id_M)$; \ \ 
$G(t) = \eta\big((G_i(t_i))_i\big)$ \ $(t = (t_i)_i \in s^\infty)$,
\vskip 1mm 
\item[] ${\cal F} = P^{-1}(0)$ \ \ and \ \ 
$F : ({\cal E}, {\cal E}^c, f_0) \to ({\cal F}, {\cal F}^c, f_0)$; \ \ 
$F(f) = f \cdot G(P(f))^{-1}$ \ $(f \in {\cal E})$.
\end{itemize}

These maps 
determine two maps $\Phi$ and $\Psi$ by  
\begin{itemize}
\item[] $\Phi : ({\cal E}, {\cal E}^c, f_0) \lra  (s^\Lambda, s^\Lambda_f, 0) \times ({\cal F}, {\cal F}^c, f_0)$; \ \ $\Phi(f) = (P(f), F(f))$,  
\vskip 1mm 
\item[] $\Psi :  (s^\Lambda, s^\Lambda_f, 0) \times ({\cal F}, {\cal F}^c, f_0) \lra ({\cal E}, {\cal E}^c, f_0)$; \ \ $\Psi(t, g) = g \cdot G(t)$. 
\end{itemize} 

The next proposition follows from Proposition~\ref{prop_stability}. 

\begin{prop}\label{prop_homeo_stability_mu} 
{\rm (1)} The maps $\Phi$ and $\Psi$ are reciplocal homeomorphisms of pairs.

{\rm (2)} {\rm (i)} If $({\cal E}_1, {\cal E}_2)$ is a subpair of $({\cal E}, {\cal E}^c)$,  
${\cal E}_1$ is ${\cal G}$-invariant and ${\cal E}_2$ is ${\cal G}^c$-invariant, then  
the map $\Phi$ restricts to the homeomorphism of the subpairs 
$$\Phi : ({\cal E}_1, {\cal E}_2) \lra (s^\Lambda, s^\Lambda_f) \times ({\cal E}_1 \cap {\cal F}, {\cal E}_2 \cap {\cal F}).$$ 
In particular, the pair $ ({\cal E}_1, {\cal E}_2)$ is $(s^\Lambda, s^\Lambda_f)$-stable.  

{\rm (ii)} If ${\cal E}_1$ is a ${\cal G}$-invariant subspace of ${\cal E}$, then 
the map $\Phi$ induces the homeomorphism of the subpairs 
$$\Phi : ({\cal E}_1, {\cal E}_1^c) \lra (s^\Lambda, s^\Lambda_f) \times ({\cal E}_1 \cap {\cal F}, {\cal E}_1^c \cap {\cal F}).$$ 
\end{prop}

\begin{example} The pair $({\cal E}(M, N), {\cal E}^c(M, N))$ is $(s^\Lambda, s^\Lambda_f)$-stable. 
Indeed, the space ${\cal E}(M, N)$ is a ${\cal G}$-invariant subspace of ${\cal E}$ and  
the map $\Phi$ induces the homeomorphism between the subpairs 
$$\Phi : ({\cal E}(M, N), {\cal E}^c(M, N)) \cong 
(s^\Lambda, s^\Lambda_f) \times ({\cal F}(M, N), {\cal F}^c(M, N)), \ \ 
\text{where ${\cal F}(M, N) = {\cal E}(M, N) \cap {\cal F}$}.$$
\end{example}

In the case where $M = N$, we can take $\id_M$ as a base point $f_0$ of the space ${\cal E}$. 
Then the support function ${\rm supp}_{f_0}$  
reduces to the ordinary support function. 
The space ${\cal E}$ includes the group ${\cal H}(M)$   
as a ${\cal G}$-invariant subspace. 
From Proposition~\ref{prop_homeo_stability_mu}\,(2) we have the following criterion. 

\begin{prop}\label{prop_homeo-gr_stability} Suppose $({\cal L}, {\cal K})$ is a pair of subgroups of ${\cal H}(M)$ with 
$({\cal G}, {\cal G}^c) \subset ({\cal L}, {\cal K}) \subset ({\cal H}(M), {\cal H}^c(M))$.
Then the map $\Phi$ induces the homeomorphism 
$$\Phi : ({\cal L}, {\cal K}) \cong 
(s^\Lambda, s^\Lambda_f) \times ({\cal L} \cap {\cal F}, {\cal K} \cap {\cal F}).$$ 
Thus, the pair $({\cal L}, {\cal K})$ is $(s^\Lambda, s^\Lambda_f)$-stable. 
\end{prop} 

The next example includes Theorem~\ref{thm_stability}\,(2). 

\begin{example} 
The pairs $({\cal H}_X(M), {\cal H}^c_X(M))$ and $({\cal H}_X(M)_i, {\cal H}_X^c(M)_i)$ $(i=0,1)$ are $(s^\Lambda, s^\Lambda_f)$-stable for any subset $X$ of $M - \cup_{i\in\Lambda}{\rm Int}\,E_i$.  
Indeed, since $({\cal G}, {\cal G}^c) \subset ({\cal H}_X(M)_1, {\cal H}_X^c(M)_1)$, 
this follows from Proposition~\ref{prop_homeo-gr_stability}. 
\end{example} 

Now we can apply 
the characterization of $(s^\infty, s^\infty_f)$-manifolds (Theorem~\ref{thm_(s,sf)}) based upon the stability (Proposition~\ref{prop_homeo-gr_stability}). 

\begin{proof}[\bf Proof of Theorem~\ref{thm_main}\,(2)] It suffices to prove that the pair $(G_0, H)$ satisfies the conditions (i) $\sim$ (iii) in Theorem~\ref{thm_(s,sf)}\,(1). 
\begin{itemize}
\item[(1)] By \cite[Corollary 1.1]{Ya2} 
the group $G_0$ is a separable completely metrizable ANR. 

\item[(2)]  
(i) By the assumption $H$ is $F_\sigma$ in $G_0$.  
(ii) Since $(G_c)_1^\ast \subset H$ and $(G_c)_1^\ast$ is homotopy dense in $G_0$ 
by \cite[Theorem 3.2]{Ya3} (and its proof), the subgroup $H$ is also homotopy dense in $G_0$. 

\item[(3)]
 From the definition of ${\cal G}$, it is seen that ${\cal G}^c \subset (G_c)_1^\ast$. 
Since $({\cal G}, {\cal G}^c) \subset (G_0, H) \subset ({\cal H}(M), {\cal H}^c(M))$, 
the $(s^\infty, s^\infty_f)$-stability of $(G_0, H)$ follows from Proposition~\ref{prop_homeo-gr_stability}.  
\end{itemize}
This completes the proof.  
\end{proof}

Finally we deduce the $(s^\infty, s^\infty_f)$-stability of groups of measure-preserving homeomorphisms.  
The next example includes Theorem~\ref{thm_stability}\,(3). 

\begin{example} 
Suppose $X$ is a subset of $M - \cup_{i\in\Lambda}{\rm Int}\,E_i$ and $\mu \in {\cal M}_g^{\cup_i(A_i \cup \partial E_i)}(M)$. 
We assume that $n_i \geq 2$ and $\mu_i = \mu|_{E_i}$ for each $i\in\Lambda$. 
Then the following pairs are $(s^\Lambda, s^\Lambda_f)$-stable: 
$$\mbox{
$({\cal H}_X(M, \mu), {\cal H}^c_X(M, \mu))$, \ $({\cal H}_X(M, \mu)_i, {\cal H}_X^c(M, \mu)_i)$ \ $(i=0,1)$ 
\ \ and \ \ 
$({\cal H}_X(M, \mu)_1, {\cal H}_X^c(M, \mu)_1^\ast)$. 
}$$ 
This follows from Proposition~\ref{prop_homeo-gr_stability}, since $({\cal G}, {\cal G}^c) \subset ({\cal H}_X(M, \mu)_1, {\cal H}_X^c(M, \mu)_1^\ast)$. 
\end{example} 


\end{document}